\documentclass[11pt,reqno]{amsart}
\usepackage{amscd,amssymb,amsmath,amsthm}
\usepackage[arrow,matrix]{xy}
\usepackage{graphicx}
\usepackage{epstopdf}
\usepackage{color}
\newtheorem{thm}{Theorem}
\newtheorem{defn}{Definition}
\newtheorem{lemma}{Lemma}

\newtheorem{rk}{Remark}

\newtheorem{ex}{Example}

\numberwithin{equation}{section} 

\begin{document}
\title [Classification  in chains of three-dimensional real evolution algebras ]
{Classification  in chains of three-dimensional real evolution algebras }

\author{ B. A. Narkuziyev, U.A. Rozikov}

\address{B.A. Narkuziyev, V.I.Romanovskiy Institute of Mathematics of Uzbek Academy of Sciences, Tashkent, Uzbekistan.}
 \email { bnarkuziev@yandex.ru}

\address{ U.Rozikov$^{a,b,c}$\begin{itemize}
 \item[$^a$] V.I.Romanovskiy Institute of Mathematics of Uzbek Academy of Sciences;
\item[$^b$] AKFA University, 1st Deadlock 10, Kukcha Darvoza, 100095, Tashkent, Uzbekistan;
\item[$^c$] Faculty of Mathematics, National University of Uzbekistan.
\end{itemize}}
\email{rozikovu@yandex.ru}

\begin{abstract} A chain of evolution algebras (CEA) is an uncountable family
(depending on time) of evolution algebras on the field of real numbers.
The matrix of structural constants of a CEA satisfies Kolmogorov-Chapman equation.
In this paper, we consider three CEAs of three-dimensional real
evolution algebras.
These CEAs depend on several (non-zero) functions
defined on the set of time.
For each chain we give full classification (up to isomorphism) of the algebras depending on the time-parameter.
We find concrete functions ensuring that the corresponding CEA contains all possible three-dimensional
evolution algebras.
\end{abstract}


 \maketitle

\section{Introduction}

An evolution algebra (EA)  (see \cite{T}) is an abstract system, it gives an insight for the study of non-Mendelian genetics.
In the relation between EAs  and Markov processes, the Chapman-Kolmogorov equation gives
the fundamental relationship between  the probability transitions (kernels). There are many
recent papers devoted to the theory of evolution algebras (see for example \cite{5}, \cite{6}, \cite{MK}, \cite{MKQ}, \cite{N}, \cite{Rbp} and the references therein).

In \cite{1} a notion of a chain of evolution algebras (CEA) is introduced.
Later,  in  \cite{7} and \cite{8} (see also \cite{Rbp}) the notion of CEA was generalized  and a concept of flow  of arbitrary finite-dimensional algebras is introduced.
It is known that each EA is determined  by a quadratic matrix of structural constants.
A CEA is a (uncountable) family of EAs depending on the two-dimensional time $(s,t)$, $0\leq s\leq t$.

The matrices of structural constants (depending on $(s,t)$) of a CEA satisfy the Chapman-Kolmogorov equation.
 In other words, a CEA is a continuous-time dynamical  system which in a fixed time  is an EA.

In \cite{12} a wide class of finite-dimensional CEAs is constructed. In \cite{1}-\cite{3}, \cite{7}-\cite{MRX}, \cite{9}-\cite{11}, \cite{12}-\cite{14}
several new classes of CEAs and flows of algebras are given. These investigations are
used to develop the theory of Markov processes of cubic matrices \cite{LR}, \cite{MRX}.
The recent book \cite{Rbp} contains a systematic presentation of these algebraic and probabilistic
approaches to study population dynamics.

In this paper we consider three CEAs constructed in  \cite{3}.

Before formulation of our main problem, let us give basic notations.
Following \cite{1} we consider a family $\{E^{[s,t]}: s,t\in \mathbb{R}, 0\leq s\leq t\}$ of
$n$-dimensional evolution algebras over the field $ \mathbb{R}$, with the basis $e_{1},e_{2},\dots,e_{n}$
and the multiplication table $$e_{i}e_{i}=\sum_{j=1}^{n}a_{i,j}^{[s,t]}e_{j},
\,\,\,\,\,\,i=1,\dots,n; \,\,\,\,\,e_{i}e_{j}=0, \,\,\,i\neq j.$$
Here parameters $s,t$ are considered as time.

Denote by $\mathcal M^{[s,t]}=\left(a_{i,j}^{[s,t]}\right)_{i,j=1,\dots n}$ the matrix of structural constants of  $ E^{[s,t]}$.
\begin{defn} A family $\{E^{[s,t]}: s,t\in \mathbb{R}, 0\leq s\leq t\}$ of
$n$-dimensional evolution algebras over the field $\mathbb R$,
is called a {\it chain of evolution algebras} (CEA) if the matrix $\mathcal M^{[s,t]}$
of structural constants satisfies the Chapman-Kolmogorov equation
\begin{equation}\label{Chapman-Kolmogorov}
\mathcal M^{[s,t]}=\mathcal M^{[s,\tau]}\mathcal M^{[\tau,t]}, \,\,\,for \,\,\,any \,\,\,s<\tau<t.
\end{equation}
\end{defn}

In this paper we study the following known CEAs (constructed in  \cite{3}):

$E_{i}^{[s,t]}$, which correspond to the $\mathcal M_{i}^{[s,t]},\,\, i=1,2,3 $ defined as:
$$\mathcal M_{1}^{[s,t]}=\frac{h(t)}{2}\left(\begin{array}{cccc}
 \frac{1}{h(s)}+f(s)&\ \ \  \frac{1}{h(s)}+f(s)&\ \ \  \frac{1}{h(s)}+f(s)\\[2mm]
 \frac{1}{h(s)}-g(s)&\ \ \  \frac{1}{h(s)}-g(s)&\ \ \  \frac{1}{h(s)}-g(s)\\[2mm]
 g(s)-f(s)&\ \ \  g(s)-f(s)&\ \ \  g(s)-f(s)\end{array}\right) $$
where $h$, $g$ and $f$ are  arbitrary functions with $h(s)\neq0$ ;
$$\mathcal M_{2}^{[s,t]}=\frac{1}{2}\left\{\begin{array}{ll}\left(\begin{array}{cccccc}
1+\psi(s) & 1+\psi(s) & 1+\psi(s)\\[2mm]
1-\varphi(s) & 1-\varphi(s) & 1-\varphi(s)\\[2mm]
\varphi(s)-\psi(s) & \varphi(s)-\psi(s) & \varphi(s)-\psi(s)\end{array}\right),\,\,\,\,\,  \mbox{if}\,\,\,\,\, s\leq t<a,\\[2mm]
\left(\begin{array}{cccccc}
0 & 0 & 0\\
0 & 0 & 0\\
0 & 0 & 0\end{array}\right),\,\,\,\,\,\,\,\,\,\,\,\,\,\,\,\,\,\,\,\,\,\,\,\,\,\,\,\,
\,\,\,\,\,\,\,\,\,\,\,\,\,\,\,\,\,\,\,\,\,\,\,\,\,\,\,\,\,\,\,\,\,\,\,\,\,\,\,\,\,\,\,
\,\,\,\,\,\,\,\,\,\,\,\,\,\,\,\, \mbox{if}\,\,\,\,\,t\geq a,
\end{array}\right. $$
where $a>0$ and $\varphi$, $\psi$ are  arbitrary functions.
$$
\mathcal M_{3}^{[s,t]}=\theta(s)\left(\begin{array}{cccccc}
 \eta(t)& \vartheta(t) & \kappa(t)\\[1mm]
 \varphi_{1}(s)\eta(t)& \varphi_{1}(s)\vartheta(t) & \varphi_{1}(s)\kappa(t)\\[2mm]
 \varphi_{2}(s)\eta(t)& \varphi_{2}(s)\vartheta(t) & \varphi_{2}(s)\kappa(t) \end{array}\right) $$
where
$$\theta(s)=\frac{1}{\eta(s)+\varphi_{1}(s)\vartheta(s)+\varphi_{2}(s)\kappa(s)}$$
and $\eta$, $\vartheta$, $\kappa$, $\varphi_{1}$, $\varphi_{2}$  are  arbitrary functions with $ \eta(s)+\varphi_{1}(s)\vartheta(s)+\varphi_{2}(s)\kappa(s)\neq0$.

\

 In \cite{4} (see also \cite{5} for the complex field case) three-dimensional real evolution algebras with $dim(E^{2})=1$ are classified
 and twelve pairwise non-isomorphic evolution algebras are described.
 They are given by the following matrices of structural constants: \\ [1mm]

 $E_{1}: \left(\begin{array}{cccccc}
\ 1 &\  1 & 0\\
-1 & -1 & 0\\
\ 0 &\ 0 & 0\end{array}\right)$,
$E_{2}: \left(\begin{array}{cccccc}
\ 1 &\ 1 & 0\\
-1 &-1 & 0\\
\ 1 & \ 1 & 0\end{array}\right)$,
$E_{3}: \left(\begin{array}{cccccc}
\ 1 & \ 1 & 0\\
-1 & -1 & 0\\
-1 & -1 & 0\end{array}\right)$,
$E_{4}: \left(\begin{array}{cccccc}
1 & 0 & 0\\
0 & 0 & 0\\
0 & 0 & 0\end{array}\right)$,\\ [1mm]
$E_{5}: \left(\begin{array}{cccccc}
1 & 0 & 0\\
0 & 0 & 0\\
1 & 0 & 0\end{array}\right)$,
$E_{6}: \left(\begin{array}{cccccc}
\ 1 & 0 & 0\\
\ 0 & 0 & 0\\
-1 & 0 & 0\end{array}\right)$,
$E_{7}: \left(\begin{array}{cccccc}
1 & 0 & 0\\
1 & 0 & 0\\
1 & 0 & 0\end{array}\right)$,
$E_{8}: \left(\begin{array}{cccccc}
\ 1 & 0 & 0\\
\ 1 & 0 & 0\\
-1 & 0 & 0\end{array}\right)$,\\ [1mm]

$E_{9}: \left(\begin{array}{cccccc}
\ 1 & 0 & 0\\
-1 & 0 & 0\\
-1 & 0 & 0\end{array}\right)$,
$E_{10}: \left(\begin{array}{cccccc}
0 & 0 & 0\\
0 & 0 & 0\\
1 & 0 & 0\end{array}\right)$,
$E_{11}: \left(\begin{array}{cccccc}
0 & 0 & 0\\
1 & 0 & 0\\
1 & 0 & 0\end{array}\right)$,
$E_{12}: \left(\begin{array}{cccccc}
\ 0 & 0 & 0\\
\ 1 & 0 & 0\\
-1 & 0 & 0\end{array}\right).$
\\

\textbf {The main problem} of this paper: For each CEA $E_{i}^{[s,t]}$, $i=1,2,3$ (listed above), to classify all EAs
involved in the CEA. That is for each algebra $E_j$, $j=1,2,\dots, 12$ to describe the set of two-dimensional times $(s,t)$
for which $E_{i}^{[s,t]}$ is isomorphic to the algebra $E_j$.

\section{On classification of algebras in chains  $ E_{i}^{[s,t]},\  i=1,2,3$.}

To give classification of algebras in  three-dimensional real chains of evolution algebras we shall prove some lemmas,
 which are important to prove the main theorems of this section.

\subsection{The case $ E_{1}^{[s,t]}$.} Any structural constants matrix of  $ E_{1}^{[s,t]}$ has the form
as in the following lemma.
\begin{lemma}
 The real evolution algebra corresponding to the matrix

$\mathcal{M}=\left(\begin{array}{cccccc}
\lambda & \lambda & \lambda\\
\mu & \mu & \mu\\
\gamma & \gamma & \gamma\end{array}\right)$ with conditions $ \lambda+\mu+\gamma\neq0$ is isomorphic  to one of the following algebras:
 \

\begin {itemize}
 \item[(a)] $E_{4}$ if one of the following conditions is hold:
 \item[1)] $\lambda\neq0, \ \mu=0, \ \gamma=0$;
 \item[2)] $\mu\neq0, \  \lambda=0, \ \gamma=0$;
 \item[3)] $\gamma\neq0, \ \lambda=0,\ \mu=0$;\\

 \item[(b)] $E_{5}$ if one of  the following conditions is hold:
  \item[1)] $\lambda=0, \ \mu\gamma>0$;
 \item[2)] $\mu=0, \ \lambda\gamma>0$;
 \item[3)] $\gamma=0, \ \lambda\mu>0$;\\

 \item[(c)] $E_{6}$ if one of  the following conditions is hold:
  \item[1)] $\lambda=0, \ \mu\gamma<0$;
 \item[2)] $\mu=0, \  \lambda\gamma<0$;
 \item[3)] $\gamma=0, \ \lambda\mu<0$;\\

 \item[(d)] $E_{7}$ if  the following conditions is hold:
 \item[] $\lambda\mu\gamma\neq 0, \ \lambda \mu(\lambda+\mu)(\lambda+\mu + \gamma)>0,\ \gamma(\lambda+\mu)>0;$\\

 \item[(e)] $E_{8}$ if one  of the following conditions is hold:
 \item[1)] $\lambda\mu\gamma\neq 0, \ \lambda \mu(\lambda+\mu)(\lambda+\mu + \gamma)>0,\ \gamma(\lambda+\mu)<0;$
 \item[2)] $\lambda\mu\gamma\neq 0, \ \lambda \mu(\lambda+\mu)(\lambda+\mu + \gamma)<0,\ \gamma(\lambda+\mu)>0;$
 \item[3)] $\lambda\mu\gamma\neq 0,  \ \lambda+\mu=0;$\\

 \item[(f)] $E_{9}$ if  the following conditions is hold:
 \item[] $\lambda\mu\gamma\neq 0, \ \lambda \mu(\lambda+\mu)(\lambda+\mu + \gamma)<0,\ \gamma(\lambda+\mu)<0.$\\
 \end {itemize}
\end{lemma}

\begin{proof}
Let $E_{\mathcal{M}}$ be the evolution algebra with basis $\{e_{1},e_{2},e_{3}\}$. Then,
the multiplication table in $E_{\mathcal{M}}$ is
\begin{equation}\label{3.1}\begin{array}{c}
e_{1}e_{1}=\lambda e_{1}+\lambda e_{2}+\lambda e_{3}, \ e_{2}e_{2}=\mu e_{1}+\mu e_{2}+\mu e_{3},\\ [1mm]
 e_{3}e_{3}=\gamma e_{1}+\gamma e_{2}+\gamma e_{3}, \ e_{1}e_{2}=e_{1}e_{3}=e_{2}e_{3}=0.
 \end{array}
\end{equation}

 Let us consider $E_{i}, \ i=1,...,12$ listed in previous section. First we consider $E_{1}$, the multiplication table in $E_{1}$ is
\begin{equation}\label{3.2} \begin{array}{c}
e_{1}'e_{1}'= e_{1}'+e_{2}',\  e_{2}'e_{2}'= -e_{1}'-e_{2}', \ e_{3}'e_{3}'= 0, \\[1mm]
e_{1}'e_{2}'=0,\ e_{1}'e_{3}'=0,\ e_{2}'e_{3}'=0.
\end{array}
\end{equation}

Assume that  $E_{\mathcal{M}}$ is isomorphic to the algebra $E_{1}$. Then there exists a change of basis as follows
\begin{equation}\label{3.3}\begin{array}{c}
 e_{1}'=x_{1}e_{1}+x_{2}e_{2}+x_{3}e_{3},\\
e_{2}'=y_{1}e_{1}+y_{2}e_{2}+y_{3}e_{3},\\
e_{3}'=z_{1}e_{1}+z_{2}e_{2}+z_{3}e_{3}
\end{array}
\end{equation}
where determinant of the change is non-zero, i.e.,
\begin{equation}\label{3.4}
x_{1}y_{2}z_{3}+x_{3}y_{1}z_{2}+x_{2}y_{3}z_{1}\neq x_{3}y_{2}z_{1}+x_{2}y_{1}z_{3}+x_{1}y_{3}z_{2}.
\end{equation}
From the equalities (\ref{3.1}) and (\ref{3.3})  we obtain the following equations:
$$ \begin{array}{c}
e_{1}'e_{1}'=(x_{1}e_{1}+x_{2}e_{2}+x_{3}e_{3})(x_{1}e_{1}+x_{2}e_{2}+x_{3}e_{3})=(\lambda x_{1}^{2}+\mu x_{2}^{2}+\gamma x_{3}^{2})(e_{1}+e_{2}+e_{3}),\\[1mm]
e_{2}'e_{2}'=(y_{1}e_{1}+y_{2}e_{2}+y_{3}e_{3})(y_{1}e_{1}+y_{2}e_{2}+y_{3}e_{3})=(\lambda y_{1}^{2}+\mu y_{2}^{2}+\gamma y_{3}^{2})(e_{1}+e_{2}+e_{3}),\\[1mm]
e_{3}'e_{3}'=(z_{1}e_{1}+z_{2}e_{2}+z_{3}e_{3})(z_{1}e_{1}+z_{2}e_{2}+z_{3}e_{3})=(\lambda z_{1}^{2}+\mu z_{2}^{2}+\gamma z_{3}^{2})(e_{1}+e_{2}+e_{3}),\\[1mm]
e_{1}'e_{2}'=(x_{1}e_{1}+x_{2}e_{2}+x_{3}e_{3})(y_{1}e_{1}+y_{2}e_{2}+y_{3}e_{3})=(\lambda x_{1}y_{1}+\mu x_{2}y_{2}+\gamma x_{3}y_{3})(e_{1}+e_{2}+e_{3}),\\[1mm]
e_{1}'e_{3}'=(x_{1}e_{1}+x_{2}e_{2}+x_{3}e_{3})(z_{1}e_{1}+z_{2}e_{2}+z_{3}e_{3})=(\lambda x_{1}z_{1}+\mu x_{2}z_{2}+\gamma x_{3}z_{3})(e_{1}+e_{2}+e_{3}),\\[1mm]
e_{2}'e_{3}'=(y_{1}e_{1}+y_{2}e_{2}+y_{3}e_{3})(z_{1}e_{1}+z_{2}e_{2}+z_{3}e_{3})=(\lambda y_{1}z_{1}+\mu y_{2}z_{2}+\gamma y_{3}z_{3})(e_{1}+e_{2}+e_{3}),\\ [1mm]
e_{1}'+e_{2}'=(x_{1}+y_{1})e_{1}+(x_{2}+y_{2})e_{2}+(x_{3}+y_{3})e_{3},\\[1mm]
-e_{1}'-e_{2}'=(-x_{1}-y_{1})e_{1}+(-x_{2}-y_{2})e_{2}+(-x_{3}-y_{3})e_{3}.
\end{array}
$$
Consequently, being $\{{e_{1}', e_{2}', e_{3}'}\}$ is the  basis of evolution algebra and from the equalities (\ref{3.2})  we have
 \begin{equation}\label{3.5}  \left\{\begin{array}{ll}
  \lambda x_{1}^{2}+\mu x_{2}^{2}+\gamma x_{3}^{2}=x_{1}+y_{1}=x_{2}+y_{2}=x_{3}+y_{3}\\
  \lambda y_{1}^{2}+\mu y_{2}^{2}+\gamma y_{3}^{2}=-x_{1}-y_{1}=-x_{2}-y_{2}=-x_{3}-y_{3}\\
  \lambda z_{1}^{2}+\mu z_{2}^{2}+\gamma z_{3}^{2}=0\\
  \lambda x_{1}y_{1}+\mu x_{2}y_{2}+\gamma x_{3}y_{3}=0\\
  \lambda x_{1}z_{1}+\mu x_{2}z_{2}+\gamma x_{3}z_{3}=0\\
  \lambda y_{1}z_{1}+\mu y_{2}z_{2}+\gamma y_{3}z_{3}=0 \ .\
  \end{array} \right.
\end{equation}
If we add the right and left sides of the first and second equations of the system (\ref{3.5}), respectively, we obtain the following equation
$$ \lambda (x_{1}^{2}+y_{1}^{2})+\mu (x_{2}^{2}+y_{2}^{2})+ \gamma (x_{3}^{2}+y_{3}^{2})=0. $$
From this equation we get
$$ \begin{array}{c}
\lambda(x_{1}+y_{1})^{2}-2\lambda x_{1}y_{1}+\mu (x_{2}+y_{2})^{2}-2\mu x_{2}y_{2}+\gamma (x_{3}+y_{3})^{2}-2\gamma x_{3}y_{3}=0 \\
 \lambda(x_{1}+y_{1})^{2}+\mu (x_{2}+y_{2})^{2}+\gamma (x_{3}+y_{3})^{2}-2(\lambda x_{1}y_{1}+\mu x_{2}y_{2}+\gamma x_{3}y_{3})=0.
 \end{array} $$
Denote  \   $ x_{1}+y_{1}=x_{2}+y_{2}=x_{3}+y_{3}=k $. If we use the fourth equation of the system (\ref{3.5}), then  this equation will be
$$ \lambda k^{2}+ \mu k^{2} +\gamma k^{2}=0, \ k^{2}(\lambda+\mu+\gamma)=0, \ k^{2}(\lambda+ \mu +\gamma)=0. $$
By the condition of the lemma $\lambda+ \mu +\gamma\neq 0 $, then $k=0$. It means that
$$ x_{1}+y_{1}=x_{2}+y_{2}=x_{3}+y_{3}=0, \ x_{1}=-y_{1}, \ x_{2}=-y_{2}, \ x_{3}=-y_{3}.$$
It is a contradiction to (\ref{3.4}), hence  $E_{\mathcal{M}}$ is not isomorphic to the algebra $E_{1}$.

Now, let us consider the algebras $E_{2}$ and $E_{3}$. For these algebras systems of equations similar to (\ref{3.5}) are  as follows, respectively
$$
\left\{\begin{array}{ll}
  \lambda x_{1}^{2}+\mu x_{2}^{2}+\gamma x_{3}^{2}=x_{1}+y_{1}=x_{2}+y_{2}=x_{3}+y_{3}\\
  \lambda y_{1}^{2}+\mu y_{2}^{2}+\gamma y_{3}^{2}=-x_{1}-y_{1}=-x_{2}-y_{2}=-x_{3}-y_{3}\\
  \lambda z_{1}^{2}+\mu z_{2}^{2}+\gamma z_{3}^{2}=x_{1}+y_{1}=x_{2}+y_{2}=x_{3}+y_{3}\\
  \lambda x_{1}y_{1}+\mu x_{2}y_{2}+\gamma x_{3}y_{3}=0\\
  \lambda x_{1}z_{1}+\mu x_{2}z_{2}+\gamma x_{3}z_{3}=0\\
  \lambda y_{1}z_{1}+\mu y_{2}z_{2}+\gamma y_{3}z_{3}=0\
  \end{array} \right.
$$
and
$$
\left\{\begin{array}{ll}
  \lambda x_{1}^{2}+\mu x_{2}^{2}+\gamma x_{3}^{2}=x_{1}+y_{1}=x_{2}+y_{2}=x_{3}+y_{3}\\
  \lambda y_{1}^{2}+\mu y_{2}^{2}+\gamma y_{3}^{2}=-x_{1}-y_{1}=-x_{2}-y_{2}=-x_{3}-y_{3}\\
  \lambda z_{1}^{2}+\mu z_{2}^{2}+\gamma z_{3}^{2}=-x_{1}-y_{1}=-x_{2}-y_{2}=-x_{3}-y_{3}\\
  \lambda x_{1}y_{1}+\mu x_{2}y_{2}+\gamma x_{3}y_{3}=0\\
  \lambda x_{1}z_{1}+\mu x_{2}z_{2}+\gamma x_{3}z_{3}=0\\
  \lambda y_{1}z_{1}+\mu y_{2}z_{2}+\gamma y_{3}z_{3}=0 \ .\
  \end{array} \right.
$$
Similar to the above, it is not difficult to show that $E_{\mathcal{M}}$ is not isomorphic to the algebras $E_{2}$ and $E_{3}$.

Let us consider the next cases.

 \textbf{Case 1.} One of $\lambda, \ \mu, \ \gamma $ is not equal to zero.

 \textbf{Case 1.1.} Let $\lambda\neq0, \ \mu=0, \ \gamma=0$. We consider the algebra $E_{4}$.
The multiplication table in $E_{4}$ is
$$ \begin{array}{c}
e_{1}'e_{1}'= e_{1}',\  e_{2}'e_{2}'=0, \ e_{3}'e_{3}'= 0, \\[1mm]
 e_{1}'e_{2}'=0,\ e_{1}'e_{3}'=0,\ e_{2}'e_{3}'=0.
\end{array}
$$
In this case the system of equations similar to (\ref{3.5}) is  as follows
$$
\left\{\begin{array}{ll}
  \lambda x_{1}^{2}=x_{1}=x_{2}=x_{3}\\
  \lambda y_{1}^{2}=0\\
  \lambda z_{1}^{2}=0\\
  \lambda x_{1}y_{1}=0\\
  \lambda x_{1}z_{1}=0\\
  \lambda y_{1}z_{1}=0\
  \end{array} \right.
$$
If we solve the first equation of this system  then there will be two solutions:

\textbf{1)} If $x_{1}=0$ then  $x_{2}=x_{3}=0$. It is contradiction to (\ref{3.4}).

\textbf{2)} If $x_{1}=\frac{1}{\lambda}$ then  $x_{2}=x_{3}=\frac{1}{\lambda}$ and $y_{1}=z_{1}=0$.
 In this case the coefficients of the change of basis (\ref{3.3}) are as follows
$ x_{1}=x_{2}=x_{3}=\frac{1}{\lambda},$ $  y_{1}=z_{1}=0 $  and
 $ y_{2}, \ y_{3}, \ z_{2}, \ z_{3}$ are arbitrary real numbers such that satisfy the condition $ y_{2}z_{3}\neq y_{3}z_{2}.$

 Without loss of  generality we can take $y_{2}=1,\ y_{3}=0, \ z_{2}=0, z_{3}=1$. So by the change of basis $ e_{1}'=\frac{1}{\lambda}e_{1}+\frac{1}{\lambda}e_{2}+\frac{1}{\lambda}e_{3}$,
 $ e_{2}'= e_{2}, \ e_{3}'= e_{3}$ we can see that the algebras $E_{\mathcal{M}}$ and $E_{4}$ are isomorphic.

\textbf{Case 1.2.}
If  $\lambda=0, \  \gamma=0, \ \mu\neq 0$  then we take the following  basis change
$\{e_{2},\  e_{1}, \ e_{3} \}$    and we have case similar to the Case 1.1.

\textbf{Case 1.3.}
If $\lambda=0, \ \mu=0, \ \gamma\neq0 $  then we take the following  basis change
 $\{e_{3},\  e_{2}, \ e_{1} \}$ and we have case similar to the Case 1.1.

\textbf{Case 2.} Two of $\lambda, \ \mu, \ \gamma $ are not equal to zero.

\textbf{Case 2.1.}
 Let $\lambda=0, \ \mu\gamma>0$. We consider the algebra $E_{5}$.  The multiplication table in $E_{5}$ is
$$ \begin{array}{c}
e_{1}'e_{1}'= e_{1}',\  e_{2}'e_{2}'=0, \ e_{3}'e_{3}'=e_{1}' , \\[1mm]
 e_{1}'e_{2}'=0,\ e_{1}'e_{3}'=0,\ e_{2}'e_{3}'=0.
\end{array}
$$
In this case the system of equations similar to (\ref{3.5}) is  as follows
\begin{equation}\label{3.6}
\left\{\begin{array}{ll}
  \mu x_{2}^{2}+\gamma x_{3}^{2}=x_{1}=x_{2}=x_{3}\\
  \mu y_{2}^{2}+\gamma y_{3}^{2}=0\\
  \mu z_{2}^{2}+\gamma z_{3}^{2}=x_{1}=x_{2}=x_{3}\\
  \mu x_{2}y_{2}+\gamma x_{3}y_{3}=0\\
  \mu x_{2}z_{2}+\gamma x_{3}z_{3}=0\\
  \mu y_{2}z_{2}+\gamma y_{3}z_{3}=0 \ .\
  \end{array} \right.
\end{equation}
From the  first equation of this system  we have the following solutions:

 \textbf{1)}  $ x_{1}=x_{2}=x_{3}=0$. It is contradiction to (\ref{3.4}).

 \textbf{2)} $ x_{1}=x_{2}=x_{3}=\frac{1}{\mu + \gamma}$. In this case from the remaining equations of (\ref{3.6}) we obtain the following system
 $$
 \left\{\begin{array}{ll}
  \mu y_{2}^{2}+\gamma y_{3}^{2}=0\\
  \mu z_{2}^{2}+\gamma z_{3}^{2}=\frac{1}{\mu + \gamma}\\
  \mu y_{2}+\gamma y_{3}=0\\
  \mu z_{2}+\gamma z_{3}=0\\
  \mu y_{2}z_{2}+\gamma y_{3}z_{3}=0.\
  \end{array} \right.
 $$
Solution of this system is \  $y_{2}=0, \ y_{3}=0, \ z_{2}=\pm \sqrt{\frac{\gamma}{\mu}}\frac{1}{\mu+\gamma}, \ z_{3}=\mp \sqrt{\frac{\mu}{\gamma}}\frac{1}{\mu+\gamma}$.

We may take the coefficients of the change of basis (\ref{3.3})  as follows
$$ \begin{array}{ll}
x_{1}=x_{2}=x_{3}=\frac{1}{\mu + \gamma}, \ y_{1}=\alpha, \ y_{2}=0,\ y_{3}=0, \\[2mm]
 z_{1}=\beta, \ z_{2}=\pm \sqrt{\frac{\gamma}{\mu}}\frac{1}{\mu+\gamma},\
z_{3}=\mp \sqrt{\frac{\mu}{\gamma}}\frac{1}{\mu+\gamma},\end{array}  $$
 where  $ \alpha\neq 0, \ \alpha,  \beta\in\mathbb{R}.$
So, we may take the change of basis
$$ e_{1}'=\frac{1}{\mu + \gamma}e_{1}+\frac{1}{\mu + \gamma}e_{2}+\frac{1}{\mu + \gamma}e_{3}, \ e_{2}'=\alpha e_{1}, \ e_{3}'=\beta e_{1}+\sqrt{\frac{\gamma}{\mu}}\frac{1}{\mu+\gamma}e_{2}-
\sqrt{\frac{\mu}{\gamma}}\frac{1}{\mu+\gamma}e_{3}  $$ and
the determinant of the matrix of this change is equal to  $D=\frac{\alpha}{(\mu + \gamma)\sqrt{\mu\gamma}}\neq 0 $. By this change of basis we can see that  $E_{\mathcal{M}}$ and $E_{5}$ are isomorphic.

\textbf{Case 2.2.}
Let  $\lambda=0, \ \mu\gamma<0$. By the condition of the lemma $\mu+\gamma\neq0$. In these case we may take the change of basis
$$ e_{1}'=\frac{1}{\mu + \gamma}e_{1}+\frac{1}{\mu + \gamma}e_{2}+\frac{1}{\mu + \gamma}e_{3}, \ e_{2}'=\alpha e_{1}, \ e_{3}'=\beta e_{1}+\sqrt{-\frac{\gamma}{\mu}}\frac{1}{\mu+\gamma}e_{2}-
\sqrt{-\frac{\mu}{\gamma}}\frac{1}{\mu+\gamma}e_{3}  $$ and
the determinant of the matrix of this change is equal to  $D=\frac{\alpha}{(\mu + \gamma)^{2}} (\sqrt{-\frac{\gamma}{\mu}}+\sqrt{-\frac{\mu}{\gamma}})\neq 0 $.
By this change of basis we can see that  $E_{\mathcal{M}}$ is isomorphic to $E_{6}$.

For the cases
\begin {itemize}
 \item[(i)]  $\mu=0, \ \lambda\gamma>0$ and $ \mu=0, \ \lambda\gamma<0$;
 \item[(ii)] $ \gamma=0 , \ \lambda\mu>0$ and $ \gamma=0 , \ \lambda\mu<0$,
\end {itemize}
respectively, if we take the  changes of  basis as follows
$\{e_{2},\  e_{1}, \ e_{3} \}$ ,  $\{e_{3},\  e_{2}, \ e_{1} \}$ then we will have cases similar to the  Case 2.1 and Case 2.2.
\

\textbf{Case 3.} None of $\lambda, \ \mu, \ \gamma $ is zero, i.e $\lambda\mu\gamma\neq 0$.

\textbf{Case 3.1.}
Let
\begin{equation}\label{3.7}
   \lambda \mu(\lambda+\mu)(\lambda+\mu + \gamma)>0, \ \gamma(\lambda+\mu)>0.
\end{equation}
We consider the algebra $E_{7}$.
The multiplication table in $E_{7}$ is
$$ \begin{array}{c}
e_{1}'e_{1}'= e_{1}',\  e_{2}'e_{2}'=e_{1}', \ e_{3}'e_{3}'=e_{1}' , \\[1mm]
 e_{1}'e_{2}'=0,\ e_{1}'e_{3}'=0,\ e_{2}'e_{3}'=0.
\end{array}
$$
In this case the system of equations similar to (\ref{3.5}) is  as follows
\begin{equation}\label{3.8}
\left\{\begin{array}{ll}
 \lambda x_{1}^{2}+ \mu x_{2}^{2}+\gamma x_{3}^{2}=x_{1}=x_{2}=x_{3}\\
  \lambda y_{1}^{2}+\mu y_{2}^{2}+\gamma y_{3}^{2}=x_{1}=x_{2}=x_{3}\\
 \lambda z_{1}^{2}+ \mu z_{2}^{2}+\gamma z_{3}^{2}=x_{1}=x_{2}=x_{3}\\
  \lambda x_{1}y_{1}+\mu x_{2}y_{2}+\gamma x_{3}y_{3}=0\\
  \lambda x_{1}z_{1}+\mu x_{2}z_{2}+\gamma x_{3}z_{3}=0\\
  \lambda y_{1}z_{1}+\mu y_{2}z_{2}+\gamma y_{3}z_{3}=0 \ .\
  \end{array} \right.
\end{equation}
If we solve the first equation of this system then:

 \textbf{1)}  $ x_{1}=x_{2}=x_{3}=0$. It is contradiction to (\ref{3.4}).

\textbf{2)} $ x_{1}=x_{2}=x_{3}=\frac{1}{\lambda+\mu + \gamma}$. In this case from the remaining equations of (\ref{3.8}) we obtain the following system
 \begin{equation}\label{3.9}
 \left\{\begin{array}{ll}
  \lambda y_{1}^{2}+\mu y_{2}^{2}+\gamma y_{3}^{2}=\frac{1}{\lambda+\mu + \gamma}\\
 \lambda z_{1}^{2}+ \mu z_{2}^{2}+\gamma z_{3}^{2}=\frac{1}{\lambda+\mu + \gamma}\\
  \lambda y_{1}+\mu y_{2}+\gamma y_{3}=0\\
  \lambda z_{1}+\mu z_{2}+\gamma z_{3}=0\\
  \lambda y_{1}z_{1}+\mu y_{2}z_{2}+\gamma y_{3}z_{3}=0 \ .\
  \end{array} \right.
 \end{equation}
Without loss of generality we solve the system (\ref{3.9}) for the case $y_{3}=0$. Then we  have
\begin{equation}\label{3.10}
\left\{\begin{array}{ll}
  \lambda y_{1}^{2}+\mu y_{2}^{2}=\frac{1}{\lambda+\mu + \gamma}\\
 \lambda y_{1}+\mu y_{2}=0\\
 \lambda z_{1}^{2}+ \mu z_{2}^{2}+\gamma z_{3}^{2}=\frac{1}{\lambda+\mu + \gamma}\\
  \lambda z_{1}+\mu z_{2}+\gamma z_{3}=0\\
  \lambda y_{1}z_{1}+\mu y_{2}z_{2}=0 \ .\
  \end{array} \right.
\end{equation}
From the second equation of the system (\ref{3.10}) we find $ y_{1}=-\frac{\mu}{\lambda}y_{2}$ and  put it in the first equation of this system then we have
the next equation
$$\frac{\mu^{2}}{\lambda}y_{2}^{2}+\mu y_{2}^{2}=\frac{1}{\lambda+\mu + \gamma}. $$
 The solution of this equation is $ y_{2}=\pm \sqrt{\frac{\lambda}{\mu(\lambda+\mu)(\lambda+\mu + \gamma)}} $. Note that, if the condition (\ref{3.7}) is hold then  $\lambda\mu>0$ .
Thus  $ y_{1}=\mp \sqrt{\frac{\mu}{\lambda(\lambda+\mu)(\lambda+\mu + \gamma)}}.$

So, we may take  $ y_{1}=- \sqrt{\frac{\mu}{\lambda(\lambda+\mu)(\lambda+\mu + \gamma)}}$ and $ y_{2}= \sqrt{\frac{\lambda}{\mu(\lambda+\mu)(\lambda+\mu + \gamma)}}$.
 From the last equation of (\ref{3.10}) we have
$$
-\lambda \sqrt{\frac{\mu}{\lambda(\lambda+\mu)(\lambda+\mu + \gamma)}} z_{1}+\mu \sqrt{\frac{\lambda}{\mu(\lambda+\mu)(\lambda+\mu + \gamma)}} z_{2}=0.
$$
 It means  that $z_{1}= z_{2}$.  From this equality and from the third and fourth equations of the system (\ref{3.10}) we have the next system of equations
$$
\left\{\begin{array}{ll}
 \lambda z_{1}^{2}+ \mu z_{1}^{2}+\gamma z_{3}^{2}=\frac{1}{\lambda+\mu + \gamma}\\
  \lambda z_{1}+\mu z_{1}+\gamma z_{3}=0.\
  \end{array} \right.
$$
Solution of this system is $ z_{1}=\pm \sqrt{\frac{\gamma}{\lambda+\mu}}\frac{1}{\lambda+\mu + \gamma}, \ z_{3}=\mp \sqrt{\frac{\lambda+\mu}{\gamma}}\frac{1}{\lambda+\mu + \gamma}.$

Thus,  one of the solution of the system (\ref{3.8}) is
$$ \begin{array}{ccc}
 x_{1}=x_{2}=x_{3}=\frac{1}{\lambda+\mu + \gamma}, \
 y_{1}=- \sqrt{\frac{\mu}{\lambda(\lambda+\mu)(\lambda+\mu + \gamma)}}, \ y_{2}= \sqrt{\frac{\lambda}{\mu(\lambda+\mu)(\lambda+\mu + \gamma)}}, \ y_{3}=0, \\[1mm]
 z_{1}=z_{2}= \sqrt{\frac{\gamma}{\lambda+\mu}}\frac{1}{\lambda+\mu + \gamma}, \ z_{3}=- \sqrt{\frac{\lambda+\mu}{\gamma}}\frac{1}{\lambda+\mu + \gamma}.
\end{array}
$$
Consequently, if  we  take the following change of basis
\begin{equation}\label{3.11}
\begin{array}{ccc}
 e_{1}'=\frac{1}{\lambda+\mu + \gamma}e_{1}+\frac{1}{\lambda+\mu + \gamma}e_{2}+\frac{1}{\lambda+\mu + \gamma}e_{3},\\[1mm]
 e_{2}'=- \sqrt{\frac{\mu}{\lambda(\lambda+\mu)(\lambda+\mu + \gamma)}}e_{1}+\sqrt{\frac{\lambda}{\mu(\lambda+\mu)(\lambda+\mu + \gamma)}}e_{2},\\[1mm]
 e_{3}'=\sqrt{\frac{\gamma}{\lambda+\mu}}\frac{1}{\lambda+\mu + \gamma}e_{1}+\sqrt{\frac{\gamma}{\lambda+\mu}}\frac{1}{\lambda+\mu + \gamma}e_{2}-
  \sqrt{\frac{\lambda+\mu}{\gamma}}\frac{1}{\lambda+\mu + \gamma}e_{3}
\end{array}
\end{equation}
then the algebra $E_{\mathcal{M}}$ will be  isomorphic to $E_{7}$ .
Indeed, according to the (\ref{3.1})
$$\begin{array}{ccc}
e_{1}'e_{1}'=\frac{1}{({\lambda+\mu + \gamma})^{2}} e_{1}^{2}+ \frac{1}{({\lambda+\mu + \gamma})^{2}} e_{2}^{2}+ \frac{1}{({\lambda+\mu + \gamma})^{2}}
e_{3}^{2}=\frac{e_{1}+e_{2}+e_{3}}{\lambda+\mu + \gamma}=e_{1}',\\[1mm]
e_{2}'e_{2}'=\frac{\mu}{\lambda(\lambda+\mu)(\lambda+\mu + \gamma)}e_{1}^{2}+\frac{\lambda}{\mu (\lambda+\mu)(\lambda+\mu + \gamma)}e_{2}^{2}=
\frac{\mu (e_{1}+e_{2}+e_{3})}{(\lambda+\mu)(\lambda+\mu + \gamma)}+\frac{\lambda (e_{1}+e_{2}+e_{3})}{ (\lambda+\mu)(\lambda+\mu + \gamma)}=e_{1}',\\[1mm]
e_{3}'e_{3}'=\frac{\gamma }{\lambda+\mu} \frac{e_{1}^{2}}{(\lambda+\mu + \gamma)^{2}}+\frac{\gamma }{\lambda+\mu} \frac{e_{2}^{2}}{(\lambda+\mu + \gamma)^{2}}
+\frac{ \lambda+\mu}{\gamma} \frac{e_{3}^{2}}{(\lambda+\mu + \gamma)^{2}}=\frac{\lambda\gamma}{(\lambda+\mu)(\lambda+\mu + \gamma)^{2}}(e_{1}+e_{2}+e_{3})+ \\[1mm]
+\frac{\mu\gamma}{(\lambda+\mu)(\lambda+\mu + \gamma)^{2}}(e_{1}+e_{2}+e_{3})+ \frac{\lambda+\mu}{(\lambda+\mu + \gamma)^{2}}(e_{1}+e_{2}+e_{3})=
\frac{e_{1}+e_{2}+e_{3}}{\lambda+\mu + \gamma}=e_{1}'\
\end{array}
$$
and $ e_{1}'e_{2}'=0,\ e_{1}'e_{3}'=0,\ e_{2}'e_{3}'=0.$\\
The determinant of the matrix of the change    (\ref{3.11}) is $D=\frac{\pm 1}{(\lambda+\mu + \gamma)\sqrt{(\lambda+\mu + \gamma)\lambda\mu\gamma}}\neq0.$

\textbf{Case 3.2.}
If
$$
  \lambda \mu(\lambda+\mu)(\lambda+\mu + \gamma)>0, \ \gamma(\lambda+\mu)<0, \ \lambda\mu>0
$$
then similarly to the above, we may take the following change of basis
\begin{equation}\label{3.12}
\begin{array}{ccc}
 e_{1}'=\frac{1}{\lambda+\mu + \gamma}e_{1}+\frac{1}{\lambda+\mu + \gamma}e_{2}+\frac{1}{\lambda+\mu + \gamma}e_{3},\\[1mm]
 e_{2}'=- \sqrt{\frac{\mu}{\lambda(\lambda+\mu)(\lambda+\mu + \gamma)}}e_{1}+\sqrt{\frac{\lambda}{\mu(\lambda+\mu)(\lambda+\mu + \gamma)}}e_{2},\\[1mm]
 e_{3}'=\sqrt{-\frac{\gamma}{\lambda+\mu}}\frac{1}{\lambda+\mu + \gamma}e_{1}+\sqrt{-\frac{\gamma}{\lambda+\mu}}\frac{1}{\lambda+\mu + \gamma}e_{2}+
  \sqrt{-\frac{\lambda+\mu}{\gamma}}\frac{1}{\lambda+\mu + \gamma}e_{3}.
\end{array}
\end{equation}
The determinant of the matrix of the change    (\ref{3.12}) is not equal to zero and for this basis it is not difficult to check that
$$ \begin{array}{c}
e_{1}'e_{1}'= e_{1}',\  e_{2}'e_{2}'=e_{1}', \ e_{3}'e_{3}'=-e_{1}' , \\[1mm]
 e_{1}'e_{2}'=0,\ e_{1}'e_{3}'=0,\ e_{2}'e_{3}'=0.
\end{array}
$$
It means  that in this case the algebra $E_{\mathcal{M}}$ is isomorphic to $E_{8}$ .

\textbf{Case 3.3.}
If
$$
 \lambda \mu(\lambda+\mu)(\lambda+\mu + \gamma)>0, \ \gamma(\lambda+\mu)<0, \  \lambda\mu<0.
$$
then we take the following change of basis
$$
\begin{array}{ccc}
 e_{1}'=\frac{1}{\lambda+\mu + \gamma}e_{1}+\frac{1}{\lambda+\mu + \gamma}e_{2}+\frac{1}{\lambda+\mu + \gamma}e_{3},\\[1mm]
 e_{2}'=\sqrt{\frac{\mu}{\lambda(\lambda+\mu)(\lambda+\mu + \gamma)}}e_{1}+\sqrt{\frac{\lambda}{\mu(\lambda+\mu)(\lambda+\mu + \gamma)}}e_{2},\\[1mm]
 e_{3}'=\sqrt{-\frac{\gamma}{\lambda+\mu}}\frac{1}{\lambda+\mu + \gamma}e_{1}+\sqrt{-\frac{\gamma}{\lambda+\mu}}\frac{1}{\lambda+\mu + \gamma}e_{2}+
  \sqrt{-\frac{\lambda+\mu}{\gamma}}\frac{1}{\lambda+\mu + \gamma}e_{3}.
\end{array}
$$
By this change of basis we can see that the algebra $E_{\mathcal{M}}$ is isomorphic to $E_{8}$ again.

\textbf{Case 3.4.}
Let
\begin{equation}\label{3.13}
  \lambda \mu(\lambda+\mu)(\lambda+\mu + \gamma)<0, \ \gamma(\lambda+\mu)>0.
\end{equation}

 Note that, if the condition (\ref{3.13}) is hold then  $\lambda\mu<0$. Then similarly to the above, in this case the algebra $E_{\mathcal{M}}$ is isomorphic to $E_{8}$ .
 For this we  take the  change of basis:
\begin{equation}\label{3.14}
\begin{array}{ccc}
 e_{1}'=\frac{1}{\lambda+\mu + \gamma}e_{1}+\frac{1}{\lambda+\mu + \gamma}e_{2}+\frac{1}{\lambda+\mu + \gamma}e_{3},\\[1mm]
 e_{2}'=\sqrt{\frac{\gamma}{\lambda+\mu}}\frac{1}{\lambda+\mu + \gamma}e_{1}+\sqrt{\frac{\gamma}{\lambda+\mu}}\frac{1}{\lambda+\mu + \gamma}e_{2}-
  \sqrt{\frac{\lambda+\mu}{\gamma}}\frac{1}{\lambda+\mu + \gamma}e_{3}, \\[1mm]
  e_{3}'=\sqrt{-\frac{\mu}{\lambda(\lambda+\mu)(\lambda+\mu + \gamma)}}e_{1}+\sqrt{-\frac{\lambda}{\mu(\lambda+\mu)(\lambda+\mu + \gamma)}}e_{2}.
\end{array}
\end{equation}
The determinant of the matrix of the change    (\ref{3.14}) is not equal to zero  and for this basis it is easy to check that
$$ \begin{array}{c}
e_{1}'e_{1}'= e_{1}',\  e_{2}'e_{2}'=e_{1}', \ e_{3}'e_{3}'=-e_{1}' , \\[1mm]
 e_{1}'e_{2}'=0,\ e_{1}'e_{3}'=0,\ e_{2}'e_{3}'=0.
\end{array}
$$

\textbf{Case 3.5.}
If
$$
  \lambda \mu(\lambda+\mu)(\lambda+\mu + \gamma)<0, \ \gamma(\lambda+\mu)<0, \ \lambda\mu>0
$$
then we take the following change of basis
\begin{equation}\label{3.15}
\begin{array}{ccc}
 e_{1}'=\frac{1}{\lambda+\mu + \gamma}e_{1}+\frac{1}{\lambda+\mu + \gamma}e_{2}+\frac{1}{\lambda+\mu + \gamma}e_{3},\\[1mm]
 e_{2}'=- \sqrt{-\frac{\mu}{\lambda(\lambda+\mu)(\lambda+\mu + \gamma)}}e_{1}+\sqrt{-\frac{\lambda}{\mu(\lambda+\mu)(\lambda+\mu + \gamma)}}e_{2},\\[1mm]
 e_{3}'=\sqrt{-\frac{\gamma}{\lambda+\mu}}\frac{1}{\lambda+\mu + \gamma}e_{1}+\sqrt{-\frac{\gamma}{\lambda+\mu}}\frac{1}{\lambda+\mu + \gamma}e_{2}+
  \sqrt{-\frac{\lambda+\mu}{\gamma}}\frac{1}{\lambda+\mu + \gamma}e_{3}.
\end{array}
\end{equation}
The determinant of the matrix of the change    (\ref{3.15}) is not equal to zero and for this basis it is not difficult to check that
$$ \begin{array}{c}
e_{1}'e_{1}'= e_{1}',\  e_{2}'e_{2}'=-e_{1}', \ e_{3}'e_{3}'=-e_{1}' , \\[1mm]
 e_{1}'e_{2}'=0,\ e_{1}'e_{3}'=0,\ e_{2}'e_{3}'=0.
\end{array}
$$
It means  that in this case the algebra $E_{\mathcal{M}}$ is isomorphic to $E_{9}$.

\textbf{Case 3.6.}
If
$$
 \lambda \mu(\lambda+\mu)(\lambda+\mu + \gamma)<0, \ \gamma(\lambda+\mu)<0, \ \lambda\mu<0
$$
then we take the following change of basis
$$
\begin{array}{ccc}
 e_{1}'=\frac{1}{\lambda+\mu + \gamma}e_{1}+\frac{1}{\lambda+\mu + \gamma}e_{2}+\frac{1}{\lambda+\mu + \gamma}e_{3},\\[1mm]
 e_{2}'= \sqrt{-\frac{\mu}{\lambda(\lambda+\mu)(\lambda+\mu + \gamma)}}e_{1}+\sqrt{-\frac{\lambda}{\mu(\lambda+\mu)(\lambda+\mu + \gamma)}}e_{2},\\[1mm]
 e_{3}'=\sqrt{-\frac{\gamma}{\lambda+\mu}}\frac{1}{\lambda+\mu + \gamma}e_{1}+\sqrt{-\frac{\gamma}{\lambda+\mu}}\frac{1}{\lambda+\mu + \gamma}e_{2}+
  \sqrt{-\frac{\lambda+\mu}{\gamma}}\frac{1}{\lambda+\mu + \gamma}e_{3}.
\end{array}
$$
By this change of basis we can see that the algebra $E_{\mathcal{M}}$ is isomorphic to $E_{9}$ again.

\textbf{Case 3.7.}
Let  $\lambda+\mu=0$.

\textbf{Case 3.7.1.}
 If $\lambda(\lambda+\gamma)>0$
then we take the following change of basis
$$
\begin{array}{ccc}
e_{1}'=\frac{1}{\gamma}e_{1} + \frac{1}{\gamma}e_{2} + \frac{1}{\gamma}e_{3},\  e_{2}'=\frac{1}{\sqrt{\lambda(\lambda+\gamma)}}e_{1}-\frac{\lambda}{\gamma}\frac{1}{\sqrt{\lambda(\lambda+\gamma)}}e_{3},\\[1mm]
e_{3}'=\frac{1}{\gamma}\sqrt{\frac{\lambda}{\lambda+\gamma}}e_{1} + \frac{\lambda+\gamma}{\lambda\gamma}\sqrt{\frac{\lambda}{\lambda+\gamma}}e_{2}+
\frac{1}{\gamma}\sqrt{\frac{\lambda}{\lambda+\gamma}}e_{3}.
\end{array}
$$
By this change of basis we can see that the algebra $E_{\mathcal{M}}$ is isomorphic to $E_{8}$.

\textbf{Case 3.7.2.}
If $\lambda(\lambda+\gamma)<0$
then we take the following change of basis
$$
\begin{array}{ccc}
e_{1}'=\frac{1}{\gamma}e_{1} + \frac{1}{\gamma}e_{2} + \frac{1}{\gamma}e_{3},\\[2mm]
e_{2}'=\frac{1}{\gamma}\sqrt{\frac{-\lambda}{\lambda+\gamma}}e_{1} + \frac{\lambda+\gamma}{\lambda\gamma}\sqrt{\frac{-\lambda}{\lambda+\gamma}}e_{2}+
\frac{1}{\gamma}\sqrt{\frac{-\lambda}{\lambda+\gamma}}e_{3},\\[2mm]
e_{3}'=\frac{1}{\sqrt{-\lambda(\lambda+\gamma)}}e_{1}-\frac{\lambda}{\gamma}\frac{1}{\sqrt{-\lambda(\lambda+\gamma)}}e_{3}.
\end{array}
$$
By this change of basis we can see that the algebra $E_{\mathcal{M}}$ is isomorphic to $E_{8}$ again.

\textbf{Case 3.7.3.}
If $\lambda+\gamma=0$ then we take the following change of basis
$$
\begin{array}{ccc}
e_{1}'=\frac{-1}{\lambda}e_{1} + \frac{-1}{\lambda}e_{2} + \frac{-1}{\lambda}e_{3},\  e_{2}'=\frac{1}{\sqrt{2}\lambda}e_{2}-\frac{1}{\sqrt{2}\lambda}e_{3},\\[1mm]
e_{3}'=\frac{\sqrt{2}}{\lambda}e_{1}+\frac{1}{\sqrt{2}\lambda}e_{2} + \frac{1}{\sqrt{2}\lambda}e_{3}.
\end{array}
$$
  In this case again the algebra $E_{\mathcal{M}}$  is isomorphic to   the algebra  $E_{8}$.

\

Now, we consider the algebra $E_{10}$. The multiplication table in $E_{10}$ is
$$
\begin{array}{c}
e_{1}'e_{1}'= 0,\  e_{2}'e_{2}'= 0, \ e_{3}'e_{3}'=e_{1}' , \\[1mm]
 e_{1}'e_{2}'=0,\ e_{1}'e_{3}'=0,\ e_{2}'e_{3}'=0.
\end{array}
$$

Assume that  $E_{\mathcal{M}}$ is isomorphic to the algebra $E_{10}$. Then there exists a change of basis as (\ref{3.3}).
In this case we have the following  system of equations similar to (\ref{3.5})
\begin{equation}\label{3.16}
\left\{\begin{array}{ll}
 \lambda x_{1}^{2}+ \mu x_{2}^{2}+\gamma x_{3}^{2}=0\\
  \lambda y_{1}^{2}+\mu y_{2}^{2}+\gamma y_{3}^{2}=0\\
 \lambda z_{1}^{2}+ \mu z_{2}^{2}+\gamma z_{3}^{2}=x_{1}=x_{2}=x_{3}\\
  \lambda x_{1}y_{1}+\mu x_{2}y_{2}+\gamma x_{3}y_{3}=0\\
  \lambda x_{1}z_{1}+\mu x_{2}z_{2}+\gamma x_{3}z_{3}=0\\
  \lambda y_{1}z_{1}+\mu y_{2}z_{2}+\gamma y_{3}z_{3}=0 \ .\
  \end{array} \right.
\end{equation}
If we use $x_{1}=x_{2}=x_{3}$ and the first equation of (\ref{3.16}) then we have the next equation
$$
x_{1}^{2}( \lambda+\mu+\gamma)=0.
$$
We know that by the condition of the lemma $ \lambda+\mu+\gamma\neq0$. Hence $x_{1}=0$. It means that $x_{1}=x_{2}=x_{3}=0$.
It is contradiction to (\ref{3.4}), hence  our assumption is not true.

In the same way, we can see that the algebra $E_{\mathcal{M}}$ is not isomorphic to the algebras $E_{11}$ and $E_{12}$.
\end{proof}

Now, using this lemma we prove the following theorem which gives classification of the  CEA:
$E_{1}^{[s,t]}$, which correspond to the $\mathcal M_{1}^{[s,t]}$.

\begin{thm} For given values $(s,t)$ of time the algebra
$E_{1}^{[s,t]}$ is isomorphic to
\

\begin {itemize}
 \item[(a)] $E_{4}$ if one of the following conditions  holds:
 \item[1)]  $  g(s)=f(s)=\frac{1}{h(s)}$;
 \item[2)]  $  g(s)=f(s)=-\frac{1}{h(s)}$;
 \item[3)]  $ g(s)=-f(s)=\frac{1}{h(s)}$;\\

 \item[(b)] $E_{5}$ if one of the following conditions  holds:
  \item[1)] $f(s)=-\frac{1}{h(s)}, \ f^{2}(s) >g^{2}(s)$;
 \item[2)]  $ g(s)=\frac{1}{h(s)}, \ g^{2}(s)>f^{2}(s) $;
 \item[3)]  $g(s)=f(s), \ \frac{1}{h^{2}(s)}>f^{2}(s)$;\\

 \item[(c)] $E_{6}$ if one of the following conditions  holds:
  \item[1)] $f(s)=-\frac{1}{h(s)}, \ f^{2}(s) <g^{2}(s)$;
 \item[2)]  $ g(s)=\frac{1}{h(s)}, \ g^{2}(s)<f^{2}(s) $;
 \item[3)]  $g(s)=f(s), \ \frac{1}{h^{2}(s)}<f^{2}(s)$;\\

 \item[(d)] $E_{7}$ if  the following conditions  hold:
 \item[] $(\frac{1}{h(s)}+f(s))(\frac{1}{h(s)}-g(s))(g(s)-f(s))\neq 0, \ (g(s)-f(s))(\frac{2}{h(s)}+f(s)-g(s))>0\\[2mm] \frac{2}{h(s)}(\frac{1}{h(s)}+f(s))(\frac{1}{h(s)}-g(s))(\frac{2}{h(s)}+f(s)-g(s))>0 $; \\

 \item[(e)] $E_{8}$ if one of the following conditions  holds:
 \item[1)] $(\frac{1}{h(s)}+f(s))(\frac{1}{h(s)}-g(s))(g(s)-f(s))\neq 0, \ (g(s)-f(s))(\frac{2}{h(s)}+f(s)-g(s))<0\\[2mm] \frac{2}{h(s)}(\frac{1}{h(s)}+f(s))(\frac{1}{h(s)}-g(s))(\frac{2}{h(s)}+f(s)-g(s))>0 $; \\

 \item[2)] $(\frac{1}{h(s)}+f(s))(\frac{1}{h(s)}-g(s))(g(s)-f(s))\neq 0, \ (g(s)-f(s))(\frac{2}{h(s)}+f(s)-g(s))>0\\[2mm] \frac{2}{h(s)}(\frac{1}{h(s)}+f(s))(\frac{1}{h(s)}-g(s))(\frac{2}{h(s)}+f(s)-g(s))<0 $; \\

 \item[3)] $(\frac{1}{h(s)}+f(s))(\frac{1}{h(s)}-g(s))(g(s)-f(s))\neq 0, \ \frac{2}{h(s)}+f(s)-g(s)=0;$\\

 \item[(f)] $E_{9}$ if  the following conditions  hold:
 \item[] $(\frac{1}{h(s)}+f(s))(\frac{1}{h(s)}-g(s))(g(s)-f(s))\neq 0, \ (g(s)-f(s))(\frac{2}{h(s)}+f(s)-g(s))<0\\[2mm] \frac{2}{h(s)}(\frac{1}{h(s)}+f(s))(\frac{1}{h(s)}-g(s))(\frac{2}{h(s)}+f(s)-g(s))<0 $.

 \end {itemize}
\end{thm}

\
\begin{proof}
We mention that in  $\mathcal M_{1}^{[s,t]}$  the parameter functions $h,\ g$ and $f$ are arbitrary with $h(s)\neq 0$.

 Let us  denote
$$\lambda(s,t)=\frac{1}{2}h(t)(\frac{1}{h(s)}+f(s)), \ \mu(s,t)=\frac{1}{2}h(t)(\frac{1}{h(s)}-g(s)), \ \gamma(s,t)=\frac{1}{2}h(t)(g(s)-f(s)).$$
 Note that
$\lambda(s,t)+\mu(s,t)+\gamma(s,t)=\frac{h(t)}{h(s)}\neq 0$.\\

a) If $  g(s)=f(s)=\frac{1}{h(s)}$ then $\mu(s,t)=\gamma(s,t)=0$  and
$\lambda(s,t)=\frac{h(t)}{h(s)}\neq 0$. By Lemma 1, it is easy to see that $E_{1}^{[s,t]}$ is isomorphic to $E_{4}$.\\

The cases $  g(s)=f(s)=-\frac{1}{h(s)}$ and   $ g(s)=-f(s)=\frac{1}{h(s)}$  are similarly.\\

b) Let us $f(s)=-\frac{1}{h(s)}, \ f^{2}(s) >g^{2}(s)$. It means that $\lambda(s,t)=0$ and
$\mu(s,t)\gamma(s,t)>0$. Indeed
$$\lambda(s,t)=\frac{1}{2}h(t)(\frac{1}{h(s)}+f(s))=\frac{1}{2}h(t)(\frac{1}{h(s)}-\frac{1}{h(s)})=0$$ and
$$\begin{array}{cc}
\mu(s,t)\gamma(s,t)=\frac{1}{4}
h^{2}(t)(\frac{1}{h(s)}-g(s))(g(s)-f(s))=\\[2mm]
=\frac{1}{4}h^{2}(t)(-f(s)-g(s))(g(s)-f(s))
=\frac{1}{4}h^{2}(t)(f^{2}(s)-g^{2}(s))>0 \ .
\end{array}
$$
In this case by  Lemma 1,  $E_{1}^{[s,t]}$ is isomorphic to $E_{5}$.

If the conditions $ g(s)=\frac{1}{h(s)}, \ g^{2}(s)>f^{2}(s)$ and $ g(s)=f(s), \ \frac{1}{h^{2}(s)}>f^{2}(s)$ are hold
then $\mu(s,t)=0, \ \lambda(s,t)\gamma(s,t)>0$ and $\gamma(s,t)=0, \ \lambda(s,t)\mu(s,t)>0$ respectively.
In these cases by Lemma 1,  $E_{1}^{[s,t]}$ is isomorphic to $E_{5}$ again.\\

c) Similar to the previous case, if $f(s)=-\frac{1}{h(s)}, \ f^{2}(s) <g^{2}(s)$ then $\lambda(s,t)=0$ and $\mu(s,t)\gamma(s,t)<0$.
If the conditions $ g(s)=\frac{1}{h(s)}, \ g^{2}(s)<f^{2}(s)$ and $ g(s)=f(s), \ \frac{1}{h^{2}(s)}<f^{2}(s)$ are hold
then $\mu(s,t)=0, \ \lambda(s,t)\gamma(s,t)<0$ and $\gamma(s,t)=0, \ \lambda(s,t)\mu(s,t)<0$ respectively.
In these cases   $E_{1}^{[s,t]}$ is isomorphic to $E_{6}$.\\

d) If
$$\begin{array}{cc}
(\frac{1}{h(s)}+f(s))(\frac{1}{h(s)}-g(s))(g(s)-f(s))\neq 0, \ (g(s)-f(s))(\frac{2}{h(s)}+f(s)-g(s))>0\\[2mm] \frac{2}{h(s)}(\frac{1}{h(s)}+f(s))(\frac{1}{h(s)}-g(s))(\frac{2}{h(s)}+f(s)-g(s))>0
\end{array}$$
 then
 $$\lambda(s,t)\mu(s,t)\gamma(s,t)=h^{3}(t)(\frac{1}{h(s)}+f(s))(\frac{1}{h(s)}-g(s))(g(s)-f(s))\neq 0 $$
 and
 $$  \gamma(s,t)(\lambda(s,t)+\mu(s,t))= \frac{1}{4}h^{2}(t)(g(s)-f(s))(\frac{2}{h(s)}+f(s)-g(s))>0,$$
  $$\begin{array}{cc}
  \lambda(s,t) \mu(s,t)(\lambda(s,t)+\mu(s,t))(\lambda(s,t)+\mu(s,t) + \gamma(s,t))=\\[2mm]
  =\frac{1}{8}\frac{h^{4}(t)}{h(s)}(\frac{1}{h(s)}+f(s))(\frac{1}{h(s)}-g(s))(\frac{2}{h(s)}+f(s)-g(s))>0.
  \end{array} $$
In this case by Lemma 1,  $E_{1}^{[s,t]}$ is isomorphic to $E_{7}$.\\

e) In the same way it is not difficult to see that
if one of the following conditions is hold
 \begin {itemize}
\item[(i)] $(\frac{1}{h(s)}+f(s))(\frac{1}{h(s)}-g(s))(g(s)-f(s))\neq 0, \ (g(s)-f(s))(\frac{2}{h(s)}+f(s)-g(s))<0\\[2mm] \frac{2}{h(s)}(\frac{1}{h(s)}+f(s))(\frac{1}{h(s)}-g(s))(\frac{2}{h(s)}+f(s)-g(s))>0 $; \\

 \item[(ii)] $(\frac{1}{h(s)}+f(s))(\frac{1}{h(s)}-g(s))(g(s)-f(s))\neq 0, \ (g(s)-f(s))(\frac{2}{h(s)}+f(s)-g(s))>0\\[2mm] \frac{2}{h(s)}(\frac{1}{h(s)}+f(s))(\frac{1}{h(s)}-g(s))(\frac{2}{h(s)}+f(s)-g(s))<0 $; \\

 \item[(iii)] $(\frac{1}{h(s)}+f(s))(\frac{1}{h(s)}-g(s))(g(s)-f(s))\neq 0, \ \frac{2}{h(s)}+f(s)-g(s)=0$
\end {itemize}
 then   $E_{1}^{[s,t]}$ is isomorphic to $E_{8}$.\\

f) If
$$\begin{array}{cc}
(\frac{1}{h(s)}+f(s))(\frac{1}{h(s)}-g(s))(g(s)-f(s))\neq 0, \ (g(s)-f(s))(\frac{2}{h(s)}+f(s)-g(s))<0\\[2mm] \frac{1}{h(s)}(\frac{1}{h(s)}+f(s))(\frac{1}{h(s)}-g(s))(\frac{2}{h(s)}+f(s)-g(s))<0
\end{array}$$
then   $E_{1}^{[s,t]}$ is isomorphic to $E_{9}$.
\end{proof}
\begin{rk}
Note that by Lemma 1, $E_{1}^{[s,t]}$ is not isomorphic to the algebras
 $ E_{1},$ $  E_{2},$ $  E_{3},$ $  E_{10},$ $ E_{11},$ $ E_{12}$
  for any $0\leq s\leq t$.
\end{rk}

To illustrate the essence of Theorem 1, consider the following example.

\begin{ex}
Let $ g(s)=4s-16, \ f(s)=4s^{2}-24s+32, \  \frac{1}{h(s)}=s+1 $ then we obtain:
\begin {itemize}
 \item[(a)]   $E_{4}$ when $s=3$;
 \item[(b)]   $E_{5}$ when $s=4$;
 \item[(c)]   $E_{6}$ when $s=\frac{11}{4}$ and $s=\frac{17}{3}$;
 \item[(d)]   $E_{7}$ when $s\in(3;4)$;
 \item[(e)]   $E_{8}$ when $s\in[0;\frac{11}{4})\cup (4;\frac{17}{3})$;
 \item[(f)]   $E_{9}$ when $s\in(\frac{11}{4};3)\cup (\frac{17}{3};\infty)$;
\end {itemize}
 where $0\leq s \leq t $ (see Fig. \ref{Fig.1}).
\end{ex}
\begin{figure}[h!]
\includegraphics[width=0.8\textwidth]{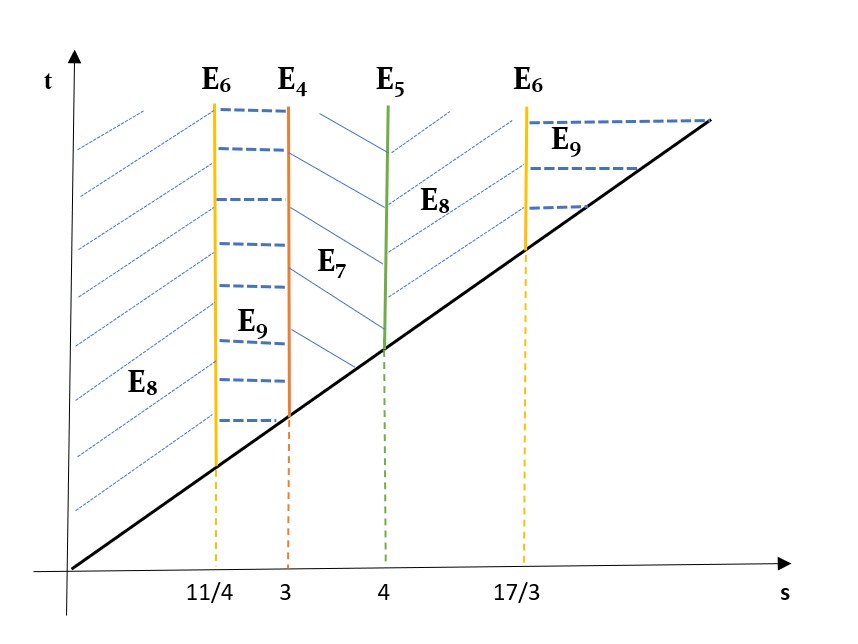}\\
\caption{The partition of the set $\{(s,t): 0\leq s\leq t\}$ corresponding to classification of EAs in the CEA $E_{1}^{[s,t]}$.}\label{Fig.1}
\end{figure}

\subsection{The case $ E_{2}^{[s,t]}$.}

In this case by Lemma 1 we get
\begin{lemma}
The real evolution algebra corresponding to the matrix

$\left(\begin{array}{cccccc}
1+a & 1+a & 1+a\\
1-b & 1-b & 1-b\\
b-a & b-a & b-a\end{array}\right)$  is isomorphic  to one of the following algebras:

\begin {itemize}
 \item[(a)] $E_{4}$ if one of the following conditions is hold:
 \item[1)]  $  a=b=1$;
 \item[2)]  $  a=b=-1$;
 \item[3)]  $ a=-1, \ b=1;$\\

 \item[(b)] $E_{5}$ if one of the following conditions is hold:
 \item[1)]  $ a=-1, \ |b|<1 $;
 \item[2)]  $ b=1, \ |a|<1 $;
 \item[3)]  $ a=b, \ |a|<1 $;\\
 \item[(c)] $E_{6}$ if one of the following conditions is hold:
  \item[1)] $ a=-1, \ |b|>1  $;
 \item[2)]  $ b=1, \ |a|>1 $;
 \item[3)]  $ a=b, \ |a|>1  $;\\

 \item[(d)] $E_{7}$  if  the following condition is hold:
 \item[]    $ (1+a)(1-b)(b-a)\neq 0, \ (1+a)(1-b)(2+a-b)>0, \ (b-a)(2+a-b)>0 $;\\

 \item[(e)] $E_{8}$ if one of the following conditions is hold:
 \item[1)] $ (1+a)(1-b)(b-a)\neq 0, \ (1+a)(1-b)(2+a-b)>0, \ (b-a)(2+a-b)<0 $;
 \item[2)] $ (1+a)(1-b)(b-a)\neq 0, \ (1+a)(1-b)(2+a-b)<0, \ (b-a)(2+a-b)>0 $;
 \item[3)] $ (1+a)(1-b)(b-a)\neq 0, \ 2+a-b=0 $;\\

 \item[(f)] $E_{9}$ if  the following conditions is hold:
 \item[ ] $ (1+a)(1-b)(b-a)\neq 0, \ (1+a)(1-b)(2+a-b)<0, \ (b-a)(2+a-b)<0 $.
 \end {itemize}
\end{lemma}

Now, by this lemma we obtain the following theorem which gives classification of the  CEA:  $E_{2}^{[s,t]}$. Let's denote by $E_{0}$ the trivial evolution algebra (i.e. with zero multiplication).
\begin{thm}
For given values $(s,t)\in \{ (s,t):0\leq s\leq t <a \}$ of time  the algebra $E_{2}^{[s,t]}$ is isomorphic to
\

\begin {itemize}
 \item[(a)] $E_{4}$ if one of the following conditions is hold:
 \item[1)]  $  \psi(s)=\varphi(s)=1$;
 \item[2)]  $  \psi(s)=\varphi(s)=-1$;
 \item[3)]  $ \psi(s)=-1, \ \varphi(s)=1;$\\

 \item[(b)] $E_{5}$ if one of the following conditions is hold:
 \item[1)]  $ \psi(s)=-1, \ |\varphi(s)|<1 $;
 \item[2)]  $ \varphi(s)=1, \ |\psi(s)|<1 $;
 \item[3)]  $ \psi(s)=\varphi(s), \ |\psi(s)|<1 $;\\

 \item[(c)] $E_{6}$ if one of the following conditions is hold:
  \item[1)] $ \psi(s)=-1, \ |\varphi(s)|>1  $;
 \item[2)]  $ \varphi(s)=1, \ |\psi(s)|>1 $;
 \item[3)]  $ \psi(s)=\varphi(s), \ |\psi(s)|>1  $;\\

 \item[(d)] $E_{7}$  if  the following condition is hold:
 \item[]    $ (1+\psi(s))(1-\varphi(s))(\varphi(s)-\psi(s))\neq 0, \\[1mm]
  (1+\psi(s))(1-\varphi(s))(2+\psi(s)-\varphi(s))>0, \ (\varphi(s)-\psi(s))(2+\psi(s)-\varphi(s))>0 $;\\

 \item[(e)] $E_{8}$ if one of the following conditions is hold:
 \item[1)]$ (1+\psi(s))(1-\varphi(s))(\varphi(s)-\psi(s))\neq 0, \\[1mm]
  (1+\psi(s))(1-\varphi(s))(2+\psi(s)-\varphi(s))>0, \ (\varphi(s)-\psi(s))(2+\psi(s)-\varphi(s))<0 $;
 \item[2)] $ (1+\psi(s))(1-\varphi(s))(\varphi(s)-\psi(s))\neq 0, \\[1mm]
  (1+\psi(s))(1-\varphi(s))(2+\psi(s)-\varphi(s))<0, \ (\varphi(s)-\psi(s))(2+\psi(s)-\varphi(s))>0 $;
 \item[3)]  $ (1+\psi(s))(1-\varphi(s))(\varphi(s)-\psi(s))\neq 0, \ 2+\psi(s)-\varphi(s)=0 $;\\

 \item[(f)] $E_{9}$ if  the following condition is hold:
 \item[1)] $ (1+\psi(s))(1-\varphi(s))(\varphi(s)-\psi(s))\neq 0, \\[1mm]
  (1+\psi(s))(1-\varphi(s))(2+\psi(s)-\varphi(s))<0, \ (\varphi(s)-\psi(s))(2+\psi(s)-\varphi(s))<0 $.
 \end {itemize}
 For all $(s,t)\in \{ (s,t):t\geq a\}$ the algebra $E_{2}^{[s,t]}$ is isomorphic to $E_{0}$.
\end{thm}

\begin{proof}
The proof of the theorem is similar to that of Theorem 1.
\end{proof}
\begin{rk}
By Lemma 2, $E_{2}^{[s,t]}$ is not isomorphic to the algebras
 $ E_{1},$ $  E_{2},$ $  E_{3},$ $  E_{10},$ $ E_{11},$ $ E_{12}$
  for any $0\leq s\leq t$.
\end{rk}
To illustrate the essence of Theorem 2, consider the following example.
\begin{ex}
Let $ \varphi(s)=s^{2}-8s+13 , \ \psi(s)=s^{2}-5 $ then for given

$(s,t)\in \{ (s,t):0\leq s\leq t <a \}$
we obtain:
\begin {itemize}
 \item[(a)]   $E_{4}$ when $s=2$;
 \item[(b)]   $E_{5}$ when $s=\frac{9}{4}$;
 \item[(c)]   $E_{6}$ when $s=6$;
 \item[(d)]   $E_{7}$ when $s\in(2;\frac{9}{4})$;
 \item[(e)]   $E_{8}$ when $s\in(\frac{9}{4};6)$;
 \item[(f)]   $E_{9}$ when $s\in[0;2)\cup (6;\infty)$;
\end {itemize}
and for all $(s,t)\in \{ (s,t):t\geq a\}$ we obtain $E_{0}$ (see Fig. \ref{Fig.2}).

\end{ex}

\begin{figure}[h!]
\includegraphics[width=0.8\textwidth]{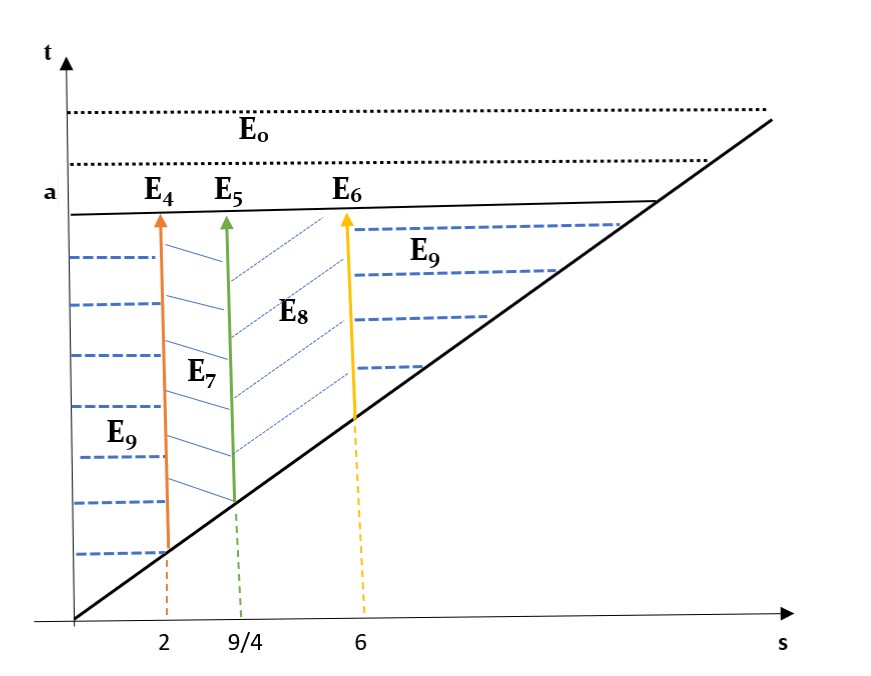}\\
\caption{The partition of time set corresponding to the classification of the CEA $E_{2}^{[s,t]}$}\label{Fig.2}
\end{figure}

Since the two different chains we considered above contain EAs isomorphic to the same algebras.
But they do not include all twelve algebras $E_i$ $i=1,2,\dots, 12$. Below we show  that $E_{3}^{[s,t]}$ is reach enough,
one can choose its parameter functions such that the chain includes all the algebras.

\subsection{The case $ E_{3}^{[s,t]}$.}

\begin{lemma}
 For the real evolution algebra corresponding to the matrix

$\left(\begin{array}{cccccc}
\alpha & \beta & \gamma\\
\lambda\alpha & \lambda\beta & \lambda\gamma\\
\mu\alpha & \mu\beta & \mu\gamma\end{array}\right)$ we have\\
I. If $\alpha\neq 0, \ \alpha^{2}+ \lambda\beta^{2}+\mu\gamma^{2}=0$ then it is isomorphic  to one of the following algebras:
 \
\begin {itemize}
 \item[(a)] $E_{1}$ if one of the following conditions is hold:
 \item[1)] $ \mu=0, \ \lambda\beta\gamma\neq 0$;
 \item[2)] $ \lambda=0, \ \beta\gamma\mu\neq 0$;
 \item[3)] $\gamma=\mu=0, \ \lambda\beta\neq 0$;
 \item[4)] $\lambda=\beta=0, \ \gamma\mu\neq 0$;\\

 \item[(b)] $E_{2}$ if one of  the following conditions is hold:
  \item[1)] $\gamma=0, \ \lambda\beta\mu\neq 0, \ \mu>0 $;
 \item[2)] $\beta=0, \ \lambda\gamma\mu\neq 0, \ \lambda>0 $;
 \item[3)] $ \lambda\beta\gamma\mu\neq 0$;\\

 \item[(c)] $E_{3}$ if one of  the following conditions is hold:
  \item[1)] $\gamma=0, \ \lambda\beta\mu\neq 0, \ \mu<0 $;
 \item[2)] $\beta=0, \  \lambda\gamma\mu\neq 0, \ \lambda<0 $.\\

  \end {itemize}

 II. If
  $\alpha\neq 0, \ \alpha^{2}+ \lambda\beta^{2}+\mu\gamma^{2}\neq 0$ then it is isomorphic  to one of the following algebras:
 \
\begin {itemize}
 \item[(d)] $E_{4}$ if  the following condition is hold:
 \item[] $ \lambda=0, \ \mu=0 $;\\

 \item[(e)] $E_{5}$ if one of  the following conditions is hold:
  \item[1)] $ \lambda=0, \ \mu>0 $;
 \item[2)] $ \mu=0,  \ \lambda>0 $;\\

 \item[(f)] $E_{6}$ if one of  the following conditions is hold:
  \item[1)] $ \lambda=0, \ \mu<0 $;
 \item[2)] $ \mu=0, \ \lambda<0 $;\\

 \item[(g)] $E_{7}$ if  the following condition is hold:
 \item[] $ \lambda>0, \ \mu>0 $;\\

 \item[(h)] $E_{8}$ if one of  the following conditions is hold:
 \item[1)] $ \lambda>0, \ \mu<0, \  \alpha^{2}+ \lambda\beta^{2}+\mu\gamma^{2}>0 $;
 \item[2)] $ \lambda<0, \ \mu>0, \  \alpha^{2}+ \lambda\beta^{2}<0, \  \alpha^{2}+ \lambda\beta^{2}+\mu\gamma^{2}>0 $;
 \item[3)] $ \lambda<0, \ \mu>0, \  \alpha^{2}+ \lambda\beta^{2}>0 $;
 \item[4)] $ \lambda<0, \ \mu<0, \  \alpha^{2}+ \lambda\beta^{2}>0, \  \alpha^{2}+ \lambda\beta^{2}+\mu\gamma^{2}<0 $;
 \item[5)] $ \lambda<0, \ \mu<0, \  \alpha^{2}+ \lambda\beta^{2}<0 $;
 \item[6)] $ \lambda\neq 0, \ \mu\neq 0, \  \alpha^{2}+ \lambda\beta^{2}=0 $;\\

 \item[(i)] $E_{9}$ if one of  the following conditions is hold:
 \item[1)] $ \lambda>0, \ \mu<0, \  \alpha^{2}+ \lambda\beta^{2}+\mu\gamma^{2}<0 $;
 \item[2)] $ \lambda<0, \ \mu>0, \  \alpha^{2}+ \lambda\beta^{2}<0, \  \alpha^{2}+ \lambda\beta^{2}+\mu\gamma^{2}<0 $;
 \item[3)] $ \lambda<0, \ \mu<0, \  \alpha^{2}+ \lambda\beta^{2}>0, \  \alpha^{2}+ \lambda\beta^{2}+\mu\gamma^{2}>0 $.\\
  \end {itemize}
III. If
 $\alpha= 0, \ \beta\neq 0 $ then it is isomorphic  to one of the following algebras:
 \
\begin {itemize}
\item[(j)] $E_{2}$ if  the following condition is hold:
 \item[] $ \mu\neq 0, \ \gamma\neq 0, \ \lambda>0, \  \lambda\beta^{2}+\mu\gamma^{2}=0 $;\\

\item[(k)] $E_{3}$ if  the following condition is hold:
 \item[] $ \mu\neq 0, \ \gamma\neq 0, \ \lambda<0, \  \lambda\beta^{2}+\mu\gamma^{2}=0 $;\\

 \item[(l)] $E_{5}$ if one of the following conditions is hold:
 \item[1)] $ \mu=0, \  \lambda>0 $;
 \item[2)] $  \lambda=0, \ \gamma\neq 0, \ \mu>0 $;\\

 \item[(m)] $E_{6}$ if one of the following conditions is hold:
 \item[1)] $ \mu=0,  \ \lambda<0 $;
 \item[2)] $ \lambda=0, \ \gamma\neq 0, \ \mu<0 $;\\

\item[(n)] $E_{7}$ if  the following condition is hold:
 \item[ ] $   \lambda>0, \ \mu>0 $;\\

\item[(o)] $E_{8}$ if one of the following conditions is hold:
 \item[1)] $  \lambda<0, \ \mu<0 $;
 \item[2)] $ \lambda\beta^{2}+\mu\gamma^{2}> 0, \  \lambda\mu<0 $;\\

\item[(p)] $E_{9}$ if  the following condition is hold:
 \item[ ] $ \lambda\beta^{2}+\mu\gamma^{2}< 0, \  \lambda\mu<0 $;\\

 \item[(q)] $E_{10}$ if  the following condition is hold:
 \item[] $ \lambda=0, \ \mu=0 $;\\

\item[(r)] $E_{11}$ if  the following condition is hold:
 \item[] $ \lambda=0, \ \gamma=0, \ \mu>0 $;\\

\item[(s)] $E_{12}$ if  the following condition is hold:
 \item[] $ \lambda=0, \ \gamma=0, \ \mu<0 $;\\

 \end {itemize}
\end{lemma}

\begin{proof}
Let  $$M_{B}=\left(\begin{array}{cccccc}
\alpha & \beta & \gamma\\
\lambda\alpha & \lambda\beta & \lambda\gamma\\
\mu\alpha & \mu\beta & \mu\gamma\end{array}\right) $$
 be the matrix of structural constants of the evolution algebra $E_{M_{B}}$  relative to a  basis $B=\{e_{1},e_{2},e_{3}\}$. Then,
the multiplication table in $E_{M_{B}}$ is
\begin{equation}\label{3.27}\begin{array}{c}
e_{1}e_{1}= \alpha e_{1}+\beta e_{2}+\gamma e_{3}, \ e_{2}e_{2}=\lambda\alpha e_{1}+\lambda\beta e_{2}+\lambda\gamma e_{3},\\ [1mm]
 e_{3}e_{3}=\mu\alpha e_{1}+\mu\beta e_{2}+\mu\gamma e_{3}, \ e_{1}e_{2}=e_{1}e_{3}=e_{2}e_{3}=0.
 \end{array}
\end{equation}
We consider all possible cases.

\textbf{Case 1.} Suppose that $\alpha\neq 0$.

\textbf{Case 1.1.} Assume that $\alpha^{2}+\beta^{2}\lambda+\gamma^{2}\mu=0$. In this case  may be one of the following cases:
\begin {itemize}
\item[1)]  $ \mu=0, \ \lambda\beta\gamma\neq 0; $
\item[2)]  $ \lambda=0, \ \beta\gamma\mu\neq 0;$
\item[3)]  $\gamma=0, \  \lambda\beta\mu\neq 0;$
\item[4)]  $\beta=0, \ \lambda\gamma\mu\neq 0;$
\item[5)]  $\gamma=\mu=0, \ \lambda\beta\neq 0;$
\item[6)]  $\lambda=\beta=0, \ \gamma\mu\neq 0;$
\item[7)]  $\lambda\beta\gamma\mu\neq 0.$
\end {itemize}
\textbf{Case 1.1.1.} If $ \mu=0, \ \lambda\beta\gamma\neq 0 $  then $\alpha^{2}+\beta^{2}\lambda=0$. This implies that $\lambda=-\frac{\alpha^{2}}{\beta^{2}}$.
In this case the matrix $M_{B}$ will be
$$ M_{B}=\left(\begin{array}{cccccc}
\alpha & \beta & \gamma\\
 -\frac{\alpha^{3}}{\beta^{2}}& -\frac{\alpha^{2}}{\beta} & -\frac{\gamma\alpha^{2}}{\beta^{2}}\\
0 & 0 & 0\end{array}\right). $$
If we take the change of basis $B'=\{ \frac{1}{\alpha}e_{1},\ \frac{\beta }{\alpha^{2}}e_{2},\  \frac{\gamma }{\alpha^{2}}e_{3} \}$ then
$$ M_{B'}=\left(\begin{array}{cccccc}
 \ 1 & \ 1 & \ 1\\
-1 & -1 & -1\\
 \ 0 & \ 0 & \ 0\end{array}\right). $$
 Let take  another change of basis $B^{''}=\{ -e_{2}, \ -e_{1}-e_{3}, \ e_{3}\}$ and  we find a structure matrix with more zeros. Then
 $$ M_{B^{''}}=\left(\begin{array}{cccccc}
 \ 1 & \ 1 &  0\\
-1 & -1 & 0\\
 \ 0 & \ 0 &  0\end{array}\right). $$
 Consequently, in this case $E_{M_{B}}$ is isomorphic to $E_{1}$.

 \textbf{Case 1.1.2.} Let  $ \lambda=0, \ \beta\gamma\mu\neq 0$ then $\alpha^{2}+\gamma^{2}\mu=0$. From this equality we find $\mu=-\frac{\alpha^{2}}{\gamma^{2}}$.
In this case the structure matrix is:
$$ M_{B}=\left(\begin{array}{cccccc}
\alpha & \beta & \gamma\\
0 & 0 & 0 \\
 -\frac{\alpha^{3}}{\gamma^{2}}& -\frac{\alpha^{2}\beta}{\gamma^{2}} & -\frac{\alpha^{2}}{\gamma}\
\end{array}\right). $$
If we take the change of basis $B'=\{ \frac{1}{\alpha}e_{1},\ \frac{\beta }{\alpha^{2}}e_{2},\  \frac{\gamma }{\alpha^{2}}e_{3} \}$ then
$$ M_{B'}=\left(\begin{array}{cccccc}
  1 &  1 &  1\\
0 & 0 & 0\\
 -1 & -1  &-1 \end{array}\right). $$
 Let $B^{''}=\{ -e_{3}, \ -e_{1}-e_{2}, \ e_{2}\}$, then by this change of basis we have
 $$ M_{B^{''}}=\left(\begin{array}{cccccc}
 \ 1 & \ 1 &  0\\
-1 & -1 & 0\\
 \ 0 & \ 0 &  0\end{array}\right). $$
 So, in this case also $E_{M_{B}}$ is isomorphic to $E_{1}$.

 \textbf{Case 1.1.3.} Let $\gamma=0, \  \lambda\beta\mu\neq 0$. Then $\alpha^{2}+\beta^{2}\lambda=0$ and  $\lambda=-\frac{\alpha^{2}}{\beta^{2}}$.
 In this case the structure matrix is:
 $$ M_{B}=\left(\begin{array}{cccccc}
\alpha & \beta & 0\\
 -\frac{\alpha^{3}}{\beta^{2}}& -\frac{\alpha^{2}}{\beta} & 0\\
\mu\alpha & \mu\beta & 0\end{array}\right). $$
Assume $\mu>0$. If we take the change of basis  $B'=\{ \frac{1}{\alpha}e_{1},\ \frac{\beta }{\alpha^{2}}e_{2},\  \frac{1 }{\alpha\sqrt{\mu}}e_{3} \}$ then
$$ M_{B'}=\left(\begin{array}{cccccc}
 \ 1 & \ 1 &  0\\
-1 & -1 & 0\\
 \ 1 & \ 1 &  0\end{array}\right). $$
Thus in this case  $E_{M_{B}}$ is isomorphic to $E_{2}$.

Assume $\mu<0$. If we take $B'=\{ \frac{1}{\alpha}e_{1},\ \frac{\beta }{\alpha^{2}}e_{2},\  \frac{1 }{\alpha\sqrt{-\mu}}e_{3} \}$ then
$$ M_{B'}=\left(\begin{array}{cccccc}
 \ 1 & \ 1 &  0\\
-1 & -1 & 0\\
 - 1 & - 1 &  0\end{array}\right). $$
 In this case  $E_{M_{B}}$ is isomorphic to $E_{3}$.

\textbf{Case 1.1.4.} Let $\beta=0, \ \lambda\gamma\mu\neq 0$. Then  $\alpha^{2}+\gamma^{2}\mu=0$ and $\mu=-\frac{\alpha^{2}}{\gamma^{2}}$.
We have
$$ M_{B}=\left(\begin{array}{cccccc}
\alpha & 0 & \gamma\\
\lambda\alpha & 0 & \lambda\gamma \\
 -\frac{\alpha^{3}}{\gamma^{2}}& 0 & -\frac{\alpha^{2}}{\gamma}\
\end{array}\right). $$
If we take the  basis $B'=\{e_{1}, \ e_{3}, \ e_{2}\}$ then
$$ M_{B'}=\left(\begin{array}{cccccc}
\alpha & \gamma & 0\\
 -\frac{\alpha^{3}}{\gamma^{2}}& -\frac{\alpha^{2}}{\gamma} & 0\\
 \lambda\alpha & \lambda\gamma & 0 \\
\end{array}\right). $$
This case is similar to the case 1.1.3. $E_{M_{B}}$ is isomorphic to $E_{2}$ when $\lambda>0$, isomorphic to $E_{3}$ when $\lambda<0$.

\textbf{Case 1.1.5.} Let $\gamma=\mu=0, \ \lambda\beta\neq 0$. Then $\alpha^{2}+\beta^{2}\lambda=0$ and  $\lambda=-\frac{\alpha^{2}}{\beta^{2}}$.
In this case the structure matrix is:
$$ M_{B}=\left(\begin{array}{cccccc}
\alpha & \beta & 0\\
 -\frac{\alpha^{3}}{\beta^{2}}& -\frac{\alpha^{2}}{\beta} & 0\\
0 & 0 & 0\end{array}\right). $$
If we take the basis $B'=\{ \frac{1}{\alpha}e_{1},\ \frac{\beta }{\alpha^{2}}e_{2},\ e_{3} \}$ then
$$ M_{B'}=\left(\begin{array}{cccccc}
 \ 1 & \ 1 &  0\\
-1 & -1 & 0\\
 \ 0 & \ 0 &  0\end{array}\right). $$
In this case  $E_{M_{B}}$ is isomorphic to $E_{1}$.

\textbf{Case 1.1.6.} Let $\lambda=\beta=0, \ \gamma\mu\neq 0 $. Taking the change of basis $B'=\{ e_{1}, \ e_{3}, \ e_{2}\}$
then  we are in the case 1.1.5.

 \textbf{Case 1.1.7.} Let $\lambda\beta\gamma\mu\neq 0.$ If we take the basis
 $B'=\{ \frac{1}{\alpha}e_{1},\ \frac{\beta }{\alpha^{2}}e_{2},\  \frac{\gamma }{\alpha^{2}}e_{3} \}$ then
$$ M_{B'}=\left(\begin{array}{cccccc}
 \ 1 & \ 1 &  1\\[1mm]
\frac{\lambda\beta^{2}}{\alpha^{2}} & \frac{\lambda\beta^{2}}{\alpha^{2}} & \frac{\lambda\beta^{2}}{\alpha^{2}}\\[1mm]
 \frac{\mu\gamma^{2}}{\alpha^{2}} & \frac{\mu\gamma^{2}}{\alpha^{2}} &  \frac{\mu\gamma^{2}}{\alpha^{2}}\end{array}\right). $$
From the equality $\alpha^{2}+\beta^{2}\lambda+\gamma^{2}\mu=0$ we have $\frac{\lambda\beta^{2}}{\alpha^{2}}=-1-\frac{\mu\gamma^{2}}{\alpha^{2}}$.
Hence
$$ M_{B'}=\left(\begin{array}{cccccc}
 \ 1 & \ 1 &  1\\[1mm]
-1-\frac{\mu\gamma^{2}}{\alpha^{2}} & -1-\frac{\mu\gamma^{2}}{\alpha^{2}} & -1-\frac{\mu\gamma^{2}}{\alpha^{2}}\\[1mm]
 \frac{\mu\gamma^{2}}{\alpha^{2}} & \frac{\mu\gamma^{2}}{\alpha^{2}} &  \frac{\mu\gamma^{2}}{\alpha^{2}}\end{array}\right). $$
Now, we take the  change of basis $B^{''}=\{ e_{1}', \ e_{2}', \ e_{3}' \}$ such that  which the matrix  is following
$$ P_{B^{''}B'}=\left(\begin{array}{cccccc}
 \frac{\alpha^{2}(K+1)}{2(\alpha^{2}+\mu\gamma^{2})} & \frac{\alpha^{2}(K-1)}{2(\alpha^{2}+\mu\gamma^{2})} &  \frac{\alpha^{2}(K+1)}{2(\alpha^{2}+\mu\gamma^{2})}\\[2mm]
 \frac{\alpha^{2}(K-1)}{2(\alpha^{2}+\mu\gamma^{2})} & \frac{\alpha^{2}(K+1)}{2(\alpha^{2}+\mu\gamma^{2})} & \frac{\alpha^{2}(K-1)}{2(\alpha^{2}+\mu\gamma^{2})}\\[2mm]
 -\frac{\mu\gamma^{2}}{\alpha^{2}} & \ 0 &  1\end{array}\right), $$
 where $K=\frac{\mu\gamma^{2}}{\alpha^{2}}(1+ \frac{\mu\gamma^{2}}{\alpha^{2}})^{2}$ and note that $|P_{B^{''}B'}|= \frac{\mu\gamma^{2}(\alpha^{2}+\mu\gamma^{2})}{\alpha^{4}}\neq 0.$
Then we obtain
$$ M_{B^{''}}=\left(\begin{array}{cccccc}
 \ 1 & \ 1 &  0\\
-1 & -1 & 0\\
 \ 1 & \ 1 &  0\end{array}\right). $$
 In this case  $E_{M_{B}}$ is isomorphic to $E_{2}$.

 \textbf{Case 1.2.} Assume that $\alpha^{2}+\beta^{2}\lambda+\gamma^{2}\mu\neq 0$.
 We consider the next cases.

 \textbf{Case 1.2.1.} Suppose $\lambda=0, \ \mu=0$. In this case the structure matrix is:
 $$ M_{B}=\left(\begin{array}{cccccc}
\alpha & \beta & \gamma\\
  0 & 0 & 0\\
0 & 0 & 0\end{array}\right). $$
Consider the change of basis $B'=\{\frac{1}{\alpha} e_{1}+\frac{\beta}{\alpha^{2}} e_{2}+\frac{\gamma}{\alpha^{2}} e_{3}; e_{2}-e_{3}; 4e_{2}+e_{3}  \}$. Then
$$ M_{B'}=\left(\begin{array}{cccccc}
 1 & 0 &  0\\
 0 & 0 & 0\\
 0 & 0 &  0\end{array}\right). $$
In this case  $E_{M_{B}}$ is isomorphic to $E_{4}$.

 \textbf{Case 1.2.2.} Suppose $\lambda=0, \ \mu\neq 0$. Then $\alpha^{2}+\gamma^{2}\mu\neq 0$. \\
 For
  $B'=\{\frac{1}{\alpha} e_{1}+\frac{\beta}{\alpha^{2}} e_{2}+\frac{\gamma}{\alpha^{2}} e_{3}; e_{2}; -\frac{\gamma\mu}{\alpha} e_{1}+e_{2}+e_{3}  \}$
the structure matrix is
$$ M_{B'}=\left(\begin{array}{cccccc}
 1+ \frac{\mu\gamma^{2}}{\alpha^{2}} & 0 &  0\\
 0 & 0 & 0\\
 \mu(\alpha^{2}+\mu\gamma^{2}) & 0 &  0\end{array}\right). $$

\textbf{Case 1.2.2.1.} Assume that $\mu>0$. Consider $B^{''}=\{ \frac{\alpha^{2}}{\alpha^{2}+\mu\gamma^{2}}e_{1};\ e_{2}; \ \frac{\alpha}{\sqrt{\mu}(\alpha^{2}+\mu\gamma^{2})}e_{3} \}$.
Then
$ M_{B^{''}}=\left(\begin{array}{cccccc}
 1 & 0 &  0\\
 0 & 0 & 0\\
 1 & 0 &  0\end{array}\right). $
In this case  $E_{M_{B}}$ is isomorphic to $E_{5}$.

\textbf{Case 1.2.2.2.} Assume that $\mu<0$. Consider $B^{''}=\{ \frac{\alpha^{2}}{\alpha^{2}+\mu\gamma^{2}}e_{1};\ e_{2}; \ \frac{\alpha}{\sqrt{-\mu}(\alpha^{2}+\mu\gamma^{2})}e_{3} \}$.
Then
$ M_{B^{''}}=\left(\begin{array}{cccccc}
 1 & 0 &  0\\
 0 & 0 & 0\\
 -1 & 0 &  0\end{array}\right). $
In this case  $E_{M_{B}}$ is isomorphic to $E_{6}$.

\textbf{Case 1.2.3.} Assume that $\lambda\neq 0$ and $\mu=0$. If we take  the basis $B'=\{ e_{1}, \ e_{3}, \ e_{2} \}$ then  we are  in the same conditions as in Case 1.2.2.
Namely, $E_{M_{B}}$ is isomorphic to $E_{5}$ and $E_{6}$ respectively $\lambda>0$ and $\lambda<0$.

\textbf{Case 1.2.4.} Assume that $\lambda\neq 0$ , $\mu\neq0$ and $\alpha^{2}+\lambda\beta^{2}\neq 0$. If $B'$ is the change of basis  such that
$$ P_{B'B}=\left(\begin{array}{cccccc}
\frac{1}{\alpha} & \frac{\beta}{\alpha^{2}} & \frac{\gamma}{\alpha^{2}}\\[1mm]
  -\frac{\lambda\beta}{\alpha} & 1 & 0\\[1mm]
\frac{-\mu\gamma}{\alpha^{2}+\lambda\beta^{2}} & \frac{-\beta\mu\gamma}{\alpha(\alpha^{2}+\lambda\beta^{2})} & \frac{1}{\alpha}\end{array}\right) $$
then we obtain
$$ M_{B'}=\left(\begin{array}{cccccc}
 \frac{1}{\alpha^{2}}(\alpha^{2}+\lambda\beta^{2}+\mu\gamma^{2}) & 0 &  0\\[1mm]
 \lambda(\alpha^{2}+\lambda\beta^{2}) & 0 & 0\\[1mm]
 \frac{\mu(\alpha^{2}+\lambda\beta^{2}+\mu\gamma^{2})}{\alpha^{2}+\lambda\beta^{2}} & 0 &  0\end{array}\right). $$
Note that $|P_{B'B}|=\frac{1}{\alpha^{4}}(\alpha^{2}+\lambda\beta^{2}+\mu\gamma^{2})\neq 0$.

\textbf{Case 1.2.4.1.} If $\lambda>0$ and $\mu>0$ then consider the change of basis $B^{''}=\{e_{1}^{''} ,\ e_{2}^{''}, \ e_{3}^{''}\}$
$$ P_{B^{''}B'}=\left(\begin{array}{cccccc}
 \frac{\alpha^{2}}{\alpha^{2}+\lambda\beta^{2}+\mu\gamma^{2}} & 0 &  0\\[2mm]
 0 & \frac{\alpha}{\sqrt{\lambda(\alpha^{2}+\lambda\beta^{2})(\alpha^{2}+\lambda\beta^{2}+\mu\gamma^{2})}} & 0\\[2mm]
 0 & \ 0 &  \frac{\alpha\sqrt{\alpha^{2}+\lambda\beta^{2}}}{\sqrt{\mu}(\alpha^{2}+\lambda\beta^{2}+\mu\gamma^{2})}\end{array}\right) $$
and the structure matrix is
$$ M_{B^{''}}=\left(\begin{array}{cccccc}
 1 & 0 &  0\\
 1 & 0 & 0\\
 1 & 0 &  0\end{array}\right). $$
In this case  $E_{M_{B}}$ is isomorphic to $E_{7}$.

\textbf{Case 1.2.4.2.} If $\lambda>0$ and $\mu<0$. Assume that $\alpha^{2}+\lambda\beta^{2}+\mu\gamma^{2}>0$.
Consider the change of basis $B^{''}=\{e_{1}^{''} ,\ e_{2}^{''}, \ e_{3}^{''}\}$
$$ P_{B^{''}B'}=\left(\begin{array}{cccccc}
 \frac{\alpha^{2}}{\alpha^{2}+\lambda\beta^{2}+\mu\gamma^{2}} & 0 &  0\\[2mm]
 0 & \frac{\alpha}{\sqrt{\lambda(\alpha^{2}+\lambda\beta^{2})(\alpha^{2}+\lambda\beta^{2}+\mu\gamma^{2})}} & 0\\[2mm]
 0 & \ 0 &  \frac{\alpha\sqrt{\alpha^{2}+\lambda\beta^{2}}}{\sqrt{-\mu}(\alpha^{2}+\lambda\beta^{2}+\mu\gamma^{2})}\end{array}\right) $$
and the structure matrix is
$$ M_{B^{''}}=\left(\begin{array}{cccccc}
 1 & 0 &  0\\
 1 & 0 & 0\\
 -1 & 0 &  0\end{array}\right). $$
In this case  $E_{M_{B}}$ is isomorphic to $E_{8}$.

\textbf{Case 1.2.4.3.} If $\lambda>0$ and $\mu<0$. Assume that $\alpha^{2}+\lambda\beta^{2}+\mu\gamma^{2}<0$.
Consider the change of basis $B^{''}=\{e_{1}^{''} ,\ e_{2}^{''}, \ e_{3}^{''}\}$
$$ P_{B^{''}B'}=\left(\begin{array}{cccccc}
 \frac{\alpha^{2}}{\alpha^{2}+\lambda\beta^{2}+\mu\gamma^{2}} & 0 &  0\\[2mm]
 0 & \frac{\alpha}{\sqrt{-\lambda(\alpha^{2}+\lambda\beta^{2})(\alpha^{2}+\lambda\beta^{2}+\mu\gamma^{2})}} & 0\\[2mm]
 0 & \ 0 &  \frac{\alpha\sqrt{\alpha^{2}+\lambda\beta^{2}}}{\sqrt{-\mu}(\alpha^{2}+\lambda\beta^{2}+\mu\gamma^{2})}\end{array}\right) $$
and the structure matrix is
$$ M_{B^{''}}=\left(\begin{array}{cccccc}
 1 & 0 &  0\\
 -1 & 0 & 0\\
 -1 & 0 &  0\end{array}\right). $$
In this case  $E_{M_{B}}$ is isomorphic to $E_{9}$.

\textbf{Case 1.2.4.4.} If $\lambda<0$ and $\mu>0$. This case is similar to the Case 1.2.4.2 and Case 1.2.4.3. There are only differences in conditions.\\
The structure matrix is   $\left(\begin{array}{cccccc}
 1 & 0 &  0\\
 1 & 0 & 0\\
 -1 & 0 &  0\end{array}\right) $  when $\alpha^{2}+\lambda\beta^{2}<0$ and $\alpha^{2}+\lambda\beta^{2}+\mu\gamma^{2}>0$.\\
 The structure matrix is   $\left(\begin{array}{cccccc}
 1 & 0 &  0\\
 -1 & 0 & 0\\
 -1 & 0 &  0\end{array}\right) $  when $\alpha^{2}+\lambda\beta^{2}<0$ and $\alpha^{2}+\lambda\beta^{2}+\mu\gamma^{2}<0$.\\
 The structure matrix is   $\left(\begin{array}{cccccc}
 1 & 0 &  0\\
 -1 & 0 & 0\\
 1 & 0 &  0\end{array}\right) $  when $\alpha^{2}+\lambda\beta^{2}>0$ and if we take the change of basis $\{e_{1}, \ e_{3}, \ e_{2} \}$ then
 we have $\left(\begin{array}{cccccc}
 1 & 0 &  0\\
 1 & 0 & 0\\
 -1 & 0 &  0\end{array}\right) $.

 \textbf{Case 1.2.4.5.} If $\lambda<0$ and $\mu<0$.
 In this case the same as above,  There are  differences in conditions.\\
The structure matrix is   $\left(\begin{array}{cccccc}
 1 & 0 &  0\\
 1 & 0 & 0\\
 -1 & 0 &  0\end{array}\right) $  when $\alpha^{2}+\lambda\beta^{2}>0$ and $\alpha^{2}+\lambda\beta^{2}+\mu\gamma^{2}<0$.\\
The structure matrix is   $\left(\begin{array}{cccccc}
 1 & 0 &  0\\
 -1 & 0 & 0\\
 -1 & 0 &  0\end{array}\right) $  when $\alpha^{2}+\lambda\beta^{2}>0$ and $\alpha^{2}+\lambda\beta^{2}+\mu\gamma^{2}>0$.\\
The structure matrix is   $\left(\begin{array}{cccccc}
 1 & 0 &  0\\
 -1 & 0 & 0\\
 1 & 0 &  0\end{array}\right) $  when $\alpha^{2}+\lambda\beta^{2}<0$ and if we take the change of basis $\{e_{1}, \ e_{3}, \ e_{2} \}$ then
 we have $\left(\begin{array}{cccccc}
 1 & 0 &  0\\
 1 & 0 & 0\\
 -1 & 0 &  0\end{array}\right) $.

\textbf{Case 1.2.5.} Suppose that $\lambda\neq 0, \ \mu\neq 0$ and $\alpha^{2}+\lambda\beta^{2}= 0$. Then  $\lambda\beta\mu\gamma\neq 0$ and
so $\lambda=-\frac{\alpha^{2}}{\beta^{2}}$.
If we take  the change of basis $B'=\{ e_{1}', \ e_{2}', \ e_{3}'  \}$  such that
$$ P_{B'B}=\left(\begin{array}{cccccc}
 \frac{1}{\alpha}& \frac{\beta}{\alpha^{2}} & \frac{\gamma}{\alpha^{2}}\\[1mm]
  -\frac{\mu}{2\alpha} & \frac{\beta\mu}{2\alpha^{2}} & \frac{1}{\gamma}\\[1mm]
\frac{2\alpha^{2}-\mu\gamma^{2}}{2\alpha^{3}} & \frac{\beta(2\alpha^{2}+\mu\gamma^{2})}{2\alpha^{4}} & \frac{\gamma}{\alpha^{2}}\end{array}\right) $$
then we obtain
$$\begin{array}{cccccc}
e_{1}'e_{1}'=\frac{1}{\alpha^{2}}e_{1}^{2}+\frac{\beta^{2}}{\alpha^{4}}e_{2}^{2}+\frac{\gamma^{2}}{\alpha^{4}}e_{3}^{2}=(\frac{1}{\alpha^{2}} +\lambda\frac{\beta^{2}}{\alpha^{4}}
+ \mu \frac{\gamma^{2}}{\alpha^{4}})(\alpha e_{1}+\beta e_{2}+\gamma e_{3})=\frac{\mu\gamma^{2}}{\alpha^{2}}e_{1}',\\[2mm]
e_{2}'e_{2}'= \frac{\mu^{2}}{4\alpha^{2}}e_{1}^{2}  + \frac{\beta^{2}\mu^{2}}{4\alpha^{4}}e_{2}^{2}   + \frac{1}{\gamma^{2}}e_{3}^{2}=
(\frac{\mu^{2}}{4\alpha^{2}} + \lambda\frac{\beta^{2}\mu^{2}}{4\alpha^{4}} + \frac{\mu}{\gamma^{2}})(\alpha e_{1}+\beta e_{2}+\gamma e_{3})=\frac{\mu\alpha^{2}}{\gamma^{2}}e_{1}',\\[2mm]
e_{3}'e_{3}'=\frac{(2\alpha^{2}-\mu\gamma^{2})^{2}}{4\alpha^{6}}e_{1}^{2}  +  \frac{\beta^{2}(2\alpha^{2}+\mu\gamma^{2})^{2}}{4\alpha^{8}}e_{2}^{2}  +
\frac{\gamma^{2}}{\alpha^{4}}e_{3}^{2}=( \frac{(2\alpha^{2}-\mu\gamma^{2})^{2}}{4\alpha^{6}}+  \lambda\frac{\beta^{2}(2\alpha^{2}+\mu\gamma^{2})^{2}}{4\alpha^{8}} + \\[2mm]
 + \mu\frac{\gamma^{2}}{\alpha^{4}})(\alpha e_{1}+\beta e_{2}+\gamma e_{3})=-\frac{\mu\gamma^{2}}{\alpha^{4}}(\alpha e_{1}+\beta e_{2}+\gamma e_{3})=
-\frac{\mu\gamma^{2}}{\alpha^{2}}e_{1}'.\\
\end{array}
$$
Moreover
$$\begin{array}{cccccc}
e_{1}'e_{2}'=-\frac{\mu}{2\alpha^{2}}e_{1}^{2}  + \frac{\beta^{2}\mu}{2\alpha^{4}}e_{2}^{2}  + \frac{1}{\alpha^{2}}e_{3}^{2}=
e_{1}^{2}(-\frac{\mu}{2\alpha^{2}}  + \lambda\frac{\beta^{2}\mu}{2\alpha^{4}} + \frac{\mu}{\alpha^{2}} )=e_{1}^{2}(\lambda\frac{\beta^{2}\mu}{2\alpha^{4}} + \frac{\mu}{2\alpha^{2}})=0\\[2mm]
e_{1}'e_{3}'=\frac{2\alpha^{2}-\mu\gamma^{2}}{2\alpha^{4}}e_{1}^{2}  + \frac{\beta^{2}(2\alpha^{2}+\mu\gamma^{2})}{2\alpha^{6}}e_{2}^{2} +
\frac{\gamma^{2}}{\alpha^{4}}e_{3}^{2}=\frac{2\alpha^{2}(\alpha^{2}+\lambda\beta^{2})+ \mu\gamma^{2}(\alpha^{2}+\lambda\beta^{2})}{2\alpha^{6}}e_{1}^{2}=0\\[2mm]
e_{2}'e_{3}'=\frac{-\mu(2\alpha^{2}-\mu\gamma^{2})}{4\alpha^{4}} e_{1}^{2} + \frac{\mu\beta^{2}(2\alpha^{2}+\mu\gamma^{2})}{4\alpha^{6}} e_{2}^{2}
+ \frac{1}{\alpha^{2}}e_{3}^{2}=e_{1}^{2}(\frac{-2\mu\alpha^{2}+\mu^{2}\gamma^{2}}{4\alpha^{4}} + \frac{2\alpha^{2}\mu\lambda\beta^{2}+\mu^{2}\gamma^{2}\lambda\beta^{2}}{4\alpha^{6}}+\\[2mm]
+\frac{\mu}{\alpha^{2}})=e_{1}^{2}(\frac{-2\mu\alpha^{4}+2\mu\alpha^{2}\lambda\beta^{2}}{4\alpha^{6}} + \frac{\mu}{\alpha^{2}})=
\frac{-2\mu\alpha^{4}+2\mu\alpha^{2}\lambda\beta^{2}+4\mu\alpha^{4}}{4\alpha^{6}}e_{1}^{2}=\frac{2\mu\alpha^{2}(\alpha^{2}+\lambda\beta^{2})}{4\alpha^{6}}=0,
\end{array}
$$
note that $|P_{B'B}|=-\frac{\beta\mu\gamma}{\alpha^{5}}\neq 0$.
So, we have
$$ M_{B'}=\left(\begin{array}{cccccc}
 \frac{\mu\gamma^{2}}{\alpha^{2}} & 0 &  0\\[1mm]
 \frac{\mu}{\gamma^{2}}\alpha^{2} & 0 & 0\\[1mm]
 -\frac{\mu\gamma^{2}}{\alpha^{2}} & 0 &  0\end{array}\right). $$
Considering the change of basis $B^{''}=\{ \frac{\alpha^{2}}{\mu\gamma^{2}}e_{1}, \ \frac{1}{\mu}e_{2}, \ \frac{-\alpha^{2}}{\mu\gamma^{2}}e_{3} \}$
and we obtain
$$ M_{B^{''}}=\left(\begin{array}{cccccc}
 \ 1 & 0 &  0\\
 \ 1 & 0 & 0\\
 -1 & 0 &  0\end{array}\right). $$
In this case  $E_{M_{B}}$ is isomorphic to $E_{8}$.

\textbf{Case 2.} Suppose that $\alpha=0$. In this case the structure matrix of the evolution algebra is
$$ M_{B}=\left(\begin{array}{cccccc}
 0 & \beta &  \gamma\\
 0 & \lambda\beta & \lambda\gamma\\
 0 & \mu\beta &  \mu\gamma\end{array}\right). $$
 Necessarily that $\beta^{2}+\gamma^{2}\neq 0$. Without loss of generality we assume $\beta\neq 0$.

\textbf{Case 2.1.} Assume $\lambda\neq 0$. Consider the change of basis $B'=\{e_{2}, \ e_{3}, \ e_{1}\}$. Then
$$ M_{B'}=\left(\begin{array}{cccccc}
 \lambda\beta & \lambda\gamma & 0\\
 \mu\beta & \mu\gamma & 0 \\
 \beta & \gamma & 0 \end{array}\right). $$

\textbf{Case 2.1.1.}  Suppose $\mu=0$, $\gamma\neq 0$ and $\lambda>0$. Consider the change of basis \\
$B^{''}=\{ e_{1}^{''}, \ e_{2}^{''}, \ e_{3}^{''} \}$  such that
$$ P_{B^{''}B'}=\left(\begin{array}{cccccc}
 \frac{1}{\lambda\beta} & \frac{\gamma}{\lambda\beta^{2}} & 0\\
 0 & 1 & 0 \\
 0 & 0 & \frac{\sqrt{\lambda}}{\lambda\beta} \end{array}\right). $$
We have
$$ M_{B^{''}}=\left(\begin{array}{cccccc}
  1 & 0 &  0\\
  0 & 0 & 0\\
  1 & 0 &  0\end{array}\right). $$
In this case  $E_{M_{B}}$ is isomorphic to $E_{5}$.

If $\lambda<0$ then we take as $B^{''}=\{ e_{1}^{''}, \ e_{2}^{''}, \ e_{3}^{''} \}$ the following
$$ P_{B^{''}B'}=\left(\begin{array}{cccccc}
 \frac{1}{\lambda\beta} & \frac{\gamma}{\lambda\beta^{2}} & 0\\
 0 & 1 & 0 \\
 0 & 0 & \frac{\sqrt{-\lambda}}{\lambda\beta} \end{array}\right). $$
We have
$$ M_{B^{''}}=\left(\begin{array}{cccccc}
  1 & 0 &  0\\
  0 & 0 & 0\\
  -1 & 0 &  0\end{array}\right). $$
In this case  $E_{M_{B}}$ is isomorphic to $E_{6}$.

\textbf{Case 2.1.2.} The cases $\mu=0, \ \gamma=0, \lambda>0$ and   $\mu=0, \ \gamma=0, \lambda<0$ are similar to the previous case.
In this case  $E_{M_{B}}$ is isomorphic to $E_{5}$  and $E_{6}$ respectively.

\textbf{Case 2.1.3.} Assume $\mu\neq 0, \ \gamma=0$. In this case if $\lambda>0, \ \mu>0$ then for
$B^{''}=\{\frac{1}{\lambda\beta}e_{1}; \ \frac{1}{\sqrt{\mu\lambda}\beta}e_{2}; \ \frac{1}{\sqrt{\lambda}\beta}e_{3}\} $ we have
$$ M_{B^{''}}=\left(\begin{array}{cccccc}
  1 & 0 &  0\\
  1 & 0 & 0\\
  1 & 0 &  0\end{array}\right). $$
In this case  $E_{M_{B}}$ is isomorphic to $E_{7}$.

Similarly, if $\lambda<0, \ \mu<0$ and $\lambda>0, \ \mu<0$ then the structure matrix will be
$$ M_{B^{''}}=\left(\begin{array}{cccccc}
  \ 1 & 0 &  0\\
  \ 1 & 0 & 0\\
  -1 & 0 &  0\end{array}\right). $$
In this case  $E_{M_{B}}$ is isomorphic to $E_{8}$.

If $\lambda<0, \ \mu>0$ then  the structure matrix will be
$$ M_{B^{''}}=\left(\begin{array}{cccccc}
  \ 1 & 0 &  0\\
  -1 & 0 & 0\\
  -1 & 0 &  0\end{array}\right). $$
In this case  $E_{M_{B}}$ is isomorphic to $E_{9}$.

\textbf{Case 2.1.4.} Assume $\mu\neq 0$, $\gamma\neq0$. Now, consider the change of basis  $B^{''}=\{\frac{1}{\lambda\beta}e_{1}; \ e_{2}; \ e_{3}\} $.
Then
$$ M_{B^{''}}=\left(\begin{array}{cccccc}
  1 & \frac{\gamma}{\lambda\beta^{2}} &  0\\
  \mu\lambda\beta^{2} & \mu\gamma & 0\\
  \lambda\beta^{2} & \gamma &  0\end{array}\right). $$

\textbf{Case 2.1.4.1.} Suppose that $\lambda\beta^{2}+\mu\gamma^{2}=0 $. And so $\mu=-\frac{\lambda\beta^{2}}{\gamma^{2}}$. It follows that
$$ M_{B^{''}}=\left(\begin{array}{cccccc}
  1 & \frac{\gamma}{\lambda\beta^{2}} &  0\\
  -\frac{\lambda^{2}\beta^{4}}{\gamma^{2}} & -\frac{\lambda\beta^{2}}{\gamma} & 0\\
  \lambda\beta^{2} & \gamma &  0\end{array}\right). $$

Let's $\lambda>0$, if we take $B^{'''}=\{ e_{1}; \   \frac{\gamma}{\lambda\beta^{2}}e_{2}, \ \frac{1}{\sqrt{\lambda}\beta}\}$ then
$ M_{B^{'''}}=\left(\begin{array}{cccccc}
  \ 1 & 1 &  0\\
  -1 & -1 & 0\\
  \ 1 & 1 &  0\end{array}\right). $\\
In this case  $E_{M_{B}}$ is isomorphic to $E_{2}$.

Let's $\lambda<0$, in this case  we may take $B^{'''}=\{ e_{1}; \   \frac{\gamma}{\lambda\beta^{2}}e_{2}, \ \frac{1}{\sqrt{\lambda}\beta}\}$,  then we have \\
$ M_{B^{'''}}=\left(\begin{array}{cccccc}
  \ 1 & 1 &  0\\
  -1 & -1 & 0\\
  -1 & -1 &  0\end{array}\right). $
In this case  $E_{M_{B}}$ is isomorphic to $E_{3}$.

\textbf{Case 2.1.4.2.} Suppose that $\lambda\beta^{2}+\mu\gamma^{2}\neq 0 $ and
 $
 \left\{\begin{array}{ll}
\lambda\mu>0\\
 \lambda\beta^{2}+\mu\gamma^{2}> 0
\end{array} \right.
 $.
 If we take $B^{'''}$ is the change of basis such that
 $$ P_{B^{'''}B^{''}}=\left(\begin{array}{cccccc}
 1 & \frac{\gamma}{\lambda\beta^{2}} & 0\\
 -\mu\gamma & 1 & 0 \\
 0 & 0 & 1 \end{array}\right) $$ then we obtain that
 $$ M_{B^{'''}}=\left(\begin{array}{cccccc}
  \frac{\lambda\beta^{2}+\mu\gamma^{2}}{\lambda\beta^{2}} & 0 &  0\\
  \mu(\lambda\beta^{2}+\mu\gamma^{2}) & 0 & 0\\
  \lambda\beta^{2} & 0 &  0\end{array}\right). $$
 Now, consider the change of basis $B^{*}=\{e_{1}^{*}, \ e_{2}^{*}, \ e_{3}^{*}  \}$
such that
$$ P_{B^{*}B^{'''}}=\left(\begin{array}{cccccc}
 \frac{\lambda\beta^{2}}{\lambda\beta^{2}+\mu\gamma^{2}} & 0 & 0\\
 0 & \sqrt{\frac{\lambda}{\mu}}\frac{\beta}{\lambda\beta^{2}+\mu\gamma^{2}} & 0 \\
 0 & 0 & \frac{1}{\sqrt{\lambda\beta^{2}+\mu\gamma^{2}}} \end{array}\right) $$
and the structure matrix is
$ M_{B^{*}}=\left(\begin{array}{cccccc}
   1 & 0 &  0\\
  1 & 0 & 0\\
  1 & 0 &  0\end{array}\right). $
In this case  $E_{M_{B}}$ is isomorphic to $E_{7}$.

Similarly, it is not difficult to show that if
 $
 \left\{\begin{array}{ll}
\lambda\mu<0\\
 \lambda\beta^{2}+\mu\gamma^{2}> 0
\end{array} \right.
 $
and $
 \left\{\begin{array}{ll}
\lambda\mu>0\\
 \lambda\beta^{2}+\mu\gamma^{2}< 0
\end{array} \right.
 $
 then the structure matrix is
 $ M_{B^{*}}=\left(\begin{array}{cccccc}
   1 & 0 &  0\\
  1 & 0 & 0\\
  -1 & 0 &  0\end{array}\right). $
In this case  $E_{M_{B}}$ is isomorphic to $E_{8}$.

If $
 \left\{\begin{array}{ll}
\lambda\mu<0\\
 \lambda\beta^{2}+\mu\gamma^{2}< 0
\end{array} \right.
 $
then the structure matrix is
$ M_{B^{*}}=\left(\begin{array}{cccccc}
   1 & 0 &  0\\
  -1 & 0 & 0\\
  -1 & 0 &  0\end{array}\right). $
In this case  $E_{M_{B}}$ is isomorphic to $E_{9}$.

\textbf{Case 2.2.} Now, we consider the case $\lambda=0$.

\textbf{Case 2.2.1.} Assume that $\mu\gamma\neq 0$. Consider the change of basis $B'\{ e_{3}, \ e_{2}, \ e_{1} \}$ and we obtain
$$ M_{B'}=\left(\begin{array}{cccccc}
  \mu\gamma & \mu\beta &  0\\
  0 & 0 & 0\\
  \gamma & \beta &  0\end{array}\right). $$
If we take $B^{''}=\{ \mu\gamma e_{1}+\mu\beta e_{2}; \ e_{2}, \ \mu e_{3}\}$  then
$$ M_{B^{*}}=\left(\begin{array}{cccccc}
   \mu^{2}\gamma^{2} & 0 &  0\\
  0 & 0 & 0\\
  \mu & 0 &  0\end{array}\right). $$
Let $\mu>0$, consider the change of basis $B^{'''}=\{ \frac{1}{\mu^{2}\gamma^{2}}e_{1};\ e_{2};\ \frac{1}{\sqrt{\mu}\mu\gamma}e_{3} \}$. Then\\
$ M_{B^{'''}}=\left(\begin{array}{cccccc}
   1 & 0 &  0\\
  0 & 0 & 0\\
  1 & 0 &  0\end{array}\right). $
In this case  $E_{M_{B}}$ is isomorphic to $E_{5}$.

Similarly, if $\mu<0$ then we have
 $ M_{B^{'''}}=\left(\begin{array}{cccccc}
   1 & 0 &  0\\
  0 & 0 & 0\\
  -1 & 0 &  0\end{array}\right). $
In this case  $E_{M_{B}}$ is isomorphic to $E_{6}$.

\textbf{Case 2.2.2.} Suppose that $\mu\gamma= 0$.

\textbf{Case 2.2.2.1.} Assume that $\mu= 0$ and $\gamma$ is an arbitrary real number. Then the structure matrix is
$ M_{B}=\left(\begin{array}{cccccc}
   0 & \beta &  \gamma\\
  0 & 0 & 0\\
  0 & 0 &  0\end{array}\right). $
 Let's  take the change of basis $B'=\{ \beta e_{2}+\gamma e_{3}; \ \frac{1}{\beta}e_{3}; \ e_{1} \}$ then
  $ M_{B'}=\left(\begin{array}{cccccc}
   0 & 0 &  0\\
  0 & 0 & 0\\
  1 & 0 &  0\end{array}\right). $
In this case  $E_{M_{B}}$ is isomorphic to $E_{10}$.

  \textbf{Case 2.2.2.2.} Assume that $\mu> 0$. It follows that $\gamma =0$. For $B'=\{ \beta e_{2};\ e_{1}; \ \frac{1}{\sqrt{\mu}}e_{3} \}$\\
  we have  $ M_{B'}=\left(\begin{array}{cccccc}
   0 & 0 &  0\\
  1 & 0 & 0\\
  1 & 0 &  0\end{array}\right). $\\
It means that in this case  $E_{M_{B}}$ is isomorphic to $E_{11}$.

  \textbf{Case 2.2.2.3.} Assume that $\mu< 0$. It follows that $\gamma =0$. For $B'=\{ \beta e_{2};\ e_{1}; \ \frac{1}{\sqrt{-\mu}}e_{3} \}$
  we have  $ M_{B'}=\left(\begin{array}{cccccc}
   0 & 0 &  0\\
  1 & 0 & 0\\
  -1 & 0 &  0\end{array}\right). $\\
It means that in this case  $E_{M_{B}}$ is isomorphic to $E_{12}$.
\end{proof}
\

Using Lemma 3 we get the following theorem.

\begin{thm}
For given values $(s,t)\in \{ (s,t):\eta(s)+\varphi_{1}(s)\vartheta(s)+\varphi_{2}(s)\kappa(s)\neq0, $ $ \eta(t)\neq 0, $
 $ \eta^{2}(t)+\varphi_{1}(s)\vartheta^{2}(t)+\varphi_{2}(s)\kappa^{2}(t)=0 \}$ of time,
 $E_{3}^{[s,t]}$ is isomorphic to
  \

\begin {itemize}
 \item[(a)] $E_{1}$ if one of the following conditions is hold:
 \item[1)] $ \varphi_{2}(s)=0, \    \varphi_{1}(s) \vartheta(t) \kappa(t)\neq 0  $;
 \item[2)] $   \varphi_{1}(s)=0, \  \varphi_{2}(s) \vartheta(t) \kappa(t)\neq 0  $;
 \item[3)] $   \kappa(t) = \varphi_{2}(s)=0, \    \varphi_{1}(s) \vartheta(t)\neq 0$;
 \item[4)] $   \varphi_{1}(s)= \vartheta(t) =0, \    \kappa(t)\varphi_{2}(s)\neq 0$;\\

 \item[(b)] $E_{2}$ if one of  the following conditions is hold:
  \item[1)] $ \kappa(t)=0, \   \varphi_{1}(s) \varphi_{2}(s) \vartheta(t)\neq 0, \  \varphi_{2}(s)>0  $;
 \item[2)] $  \vartheta(t)=0, \    \varphi_{1}(s) \varphi_{2}(s)\kappa(t)\neq 0, \   \varphi_{1}(s)>0  $;
 \item[3)] $   \varphi_{1}(s) \varphi_{2}(s) \vartheta(t) \kappa(t)\neq 0   $;\\

 \item[(c)] $E_{3}$ if one of  the following conditions is hold:
   \item[1)] $ \kappa(t)=0, \   \varphi_{1}(s) \varphi_{2}(s) \vartheta(t)\neq 0, \  \varphi_{2}(s)<0  $;
 \item[2)] $  \vartheta(t)=0, \    \varphi_{1}(s) \varphi_{2}(s)\kappa(t)\neq 0, \   \varphi_{1}(s)<0  $;\\

  \end {itemize}

For given values $(s,t)\in \{ (s,t):\eta(s)+\varphi_{1}(s)\vartheta(s)+\varphi_{2}(s)\kappa(s)\neq0, $ $ \eta(t)\neq 0, $
$  \eta^{2}(t)+\varphi_{1}(s)\vartheta^{2}(t)+\varphi_{2}(s)\kappa^{2}(t)\neq 0 \}$ of time,
 $E_{3}^{[s,t]}$ is isomorphic to

  \

\begin {itemize}
 \item[(d)] $E_{4}$ if  the following condition is hold:
 \item[] $   \varphi_{1}(s)=0, \  \varphi_{2}(s)=0 $;\\

 \item[(e)] $E_{5}$ if one of  the following conditions is hold:
  \item[1)] $ \varphi_{1}(s)=0, \     \varphi_{2}(s)>0  $;
 \item[2)] $    \varphi_{2}(s)=0, \  \varphi_{1}(s)>0 $;\\

 \item[(f)] $E_{6}$ if one of  the following conditions is hold:
  \item[1)] $ \varphi_{1}(s)=0, \ \varphi_{2}(s)<0   $;
 \item[2)] $ \varphi_{2}(s)=0, \ \varphi_{1}(s)<0 $;\\

 \item[(g)] $E_{7}$ if  the following condition is hold:
 \item[] $ \varphi_{1}(s)>0, \ \varphi_{2}(s)>0 $;\\

 \item[(h)] $E_{8}$ if one of  the following conditions is hold:
 \item[1)] $ \varphi_{1}(s)>0, \ \varphi_{2}(s)<0, \   \eta^{2}(t)+\varphi_{1}(s)\vartheta^{2}(t)+\varphi_{2}(s)\kappa^{2}(t)>0 $;
 \item[2)] $ \varphi_{1}(s)<0, \ \varphi_{2}(s) >0, \  \eta^{2}(t)+\varphi_{1}(s)\vartheta^{2}(t)<0, \   \eta^{2}(t)+\varphi_{1}(s)\vartheta^{2}(t)+\varphi_{2}(s)\kappa^{2}(t)>0 $;
 \item[3)] $ \varphi_{1}(s)<0, \ \varphi_{2}(s) >0, \  \eta^{2}(t)+\varphi_{1}(s)\vartheta^{2}(t)>0 $;
 \item[4)] $ \varphi_{1}(s)<0, \ \varphi_{2}(s) <0, \  \eta^{2}(t)+\varphi_{1}(s)\vartheta^{2}(t)>0, \   \eta^{2}(t)+\varphi_{1}(s)\vartheta^{2}(t)+\varphi_{2}(s)\kappa^{2}(t)<0 $;
 \item[5)] $ \varphi_{1}(s)<0, \ \varphi_{2}(s) <0, \  \eta^{2}(t)+\varphi_{1}(s)\vartheta^{2}(t)<0 $;
 \item[6)] $ \varphi_{1}(s)\neq 0, \ \varphi_{2}(s) \neq 0, \  \eta^{2}(t)+\varphi_{1}(s)\vartheta^{2}(t)=0 $;\\

 \item[(i)] $E_{9}$ if one of  the following conditions is hold:
 \item[1)] $ \varphi_{1}(s)>0, \ \varphi_{2}(s)<0, \  \eta^{2}(t)+\varphi_{1}(s)\vartheta^{2}(t)+\varphi_{2}(s)\kappa^{2}(t)<0 $;
 \item[2)] $ \varphi_{1}(s)<0, \ \varphi_{2}(s)>0, \  \eta^{2}(t)+\varphi_{1}(s)\vartheta^{2}(t)<0, \  \eta^{2}(t)+\varphi_{1}(s)\vartheta^{2}(t)+\varphi_{2}(s)\kappa^{2}(t)<0 $;
 \item[3)] $ \varphi_{1}(s)<0, \ \varphi_{2}(s)<0, \  \eta^{2}(t)+\varphi_{1}(s)\vartheta^{2}(t)>0, \  \eta^{2}(t)+\varphi_{1}(s)\vartheta^{2}(t)+\varphi_{2}(s)\kappa^{2}(t)>0 $.\\

  \end {itemize}

For given values $(s,t)\in \{ (s,t):\eta(s)+\varphi_{1}(s)\vartheta(s)+\varphi_{2}(s)\kappa(s)\neq0, \ \eta(t)=0 , \  \vartheta(t)\neq 0 \}$ of time ,
 $E_{3}^{[s,t]}$ is isomorphic to
  \

\begin {itemize}
\item[(j)] $E_{2}$ if  the following condition is hold:
 \item[] $ \varphi_{2}(s) \neq 0, \ \kappa(t)\neq 0, \ \varphi_{1}(s)>0, \  \varphi_{1}(s)\vartheta^{2}(t)+\varphi_{2}(s)\kappa^{2}(t)=0 $;\\

\item[(k)] $E_{3}$ if  the following condition is hold:
 \item[] $ \varphi_{2}(s)\neq 0, \ \kappa(t)\neq 0, \ \varphi_{1}(s)<0, \  \varphi_{1}(s)\vartheta^{2}(t)+\varphi_{2}(s)\kappa^{2}(t)=0 $;\\

 \item[(l)] $E_{5}$ if one of the following conditions is hold:
 \item[1)] $ \varphi_{2}(s)=0, \  \varphi_{1}(s) >0 $;
 \item[2)] $  \varphi_{1}(s) =0, \kappa(t)\neq 0, \ \varphi_{2}(s)>0 $;\\

 \item[(m)] $E_{6}$ if one of the following conditions is hold:
 \item[1)] $ \varphi_{2}(s)=0, \  \varphi_{1}(s) <0 $;
 \item[2)] $ \varphi_{1}(s) =0, \ \kappa(t)\neq 0, \ \varphi_{2}(s)<0 $;\\

\item[(n)] $E_{7}$ if  the following condition is hold:
 \item[ ] $  \varphi_{1}(s) >0, \ \varphi_{2}(s)>0 $;\\

\item[(o)] $E_{8}$ if one of the following conditions is hold:
 \item[1)] $  \varphi_{1}(s)<0, \ \varphi_{2}(s)<0 $;
 \item[2)] $ \varphi_{1}(s)\vartheta^{2}(t)+\varphi_{2}(s)\kappa^{2}(t)> 0, \  \varphi_{1}(s)\varphi_{2}(s)<0 $;\\

\item[(p)] $E_{9}$ if  the following condition is hold:
 \item[ ] $ \varphi_{1}(s)\vartheta^{2}(t)+\varphi_{2}(s)\kappa^{2}(t)< 0, \  \varphi_{1}(s)\varphi_{2}(s)<0 $;\\

 \item[(q)] $E_{10}$ if  the following condition is hold:
 \item[] $ \varphi_{1}(s)=0, \ \varphi_{2}(s)=0 $;\\

\item[(r)] $E_{11}$ if  the following condition is hold:
 \item[] $ \varphi_{1}(s)=0, \ \kappa(t) =0, \ \varphi_{2}(s)>0 $;\\

\item[(s)] $E_{12}$ if  the following condition is hold:
 \item[] $ \varphi_{1}(s)=0, \ \kappa(t) =0, \ \varphi_{2}(s)<0 $;\\

 \end {itemize}
\end{thm}
\begin{proof}
The proof follows from Lemma 3.
\end{proof}
\

To illustrate the essence of Theorem 3, consider the following example.
\begin{ex}
Let
 $ \eta(t)=\left\{\begin{array}{ll}  t+1, \ 0\leq t < 6 \\ [1mm]
 0,\ \ \ \ \ \  t\geq 6  \end{array} \right. $,
 $ \vartheta(t)=\left\{\begin{array}{ll} \frac{1}{\sqrt{2}}(t+1), \ 0\leq t<2\\ [2mm]
 \sqrt{(t-1)^{2}+4}, \ 2\leq t <3 \\ [1mm]
 0, \ \ \ \ \ \ 3\leq t < 6 \\ [1mm]
  t-2,\ \   t\geq 6   \end{array} \right. $,

  $ \kappa(t)=\left\{\begin{array}{ll}  0, \ \ \ \ \ \ 0\leq t < 1 \\ [1mm]
 t-3,\ 1\leq t<6 \\ [1mm]
 0, \ \ \ \ \ \  t\geq 6  \end{array} \right. $,
 $ \varphi_{1}(s)=\left\{\begin{array}{ll}  -2, \ \ \  \ 0\leq s < 3 \\ [1mm]
 0,\ \ \ \ \ \  3\leq s \leq 5 \\ [1mm]
 s-5,7 \ ,  \  s>5  \end{array} \right. $,

  $ \varphi_{2}(s)=\left\{\begin{array}{ll}  -1, \ \ \  \ 0\leq s < 1 \\ [1mm]
 0,\ \ \ \ \ \  1\leq s <2 \\ [1mm]
 1`, \ \ \ \ \ \ 2\leq s <3 \\[1mm]
 s-4  ,  \  s \geq 3  \end{array} \right. $.

 \begin{figure}[h!]
\includegraphics[width=0.7\textwidth]{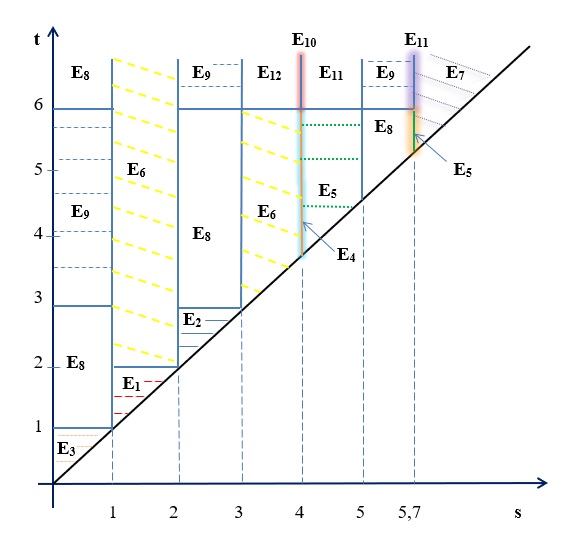}\\
\caption{The partition corresponding to the classification of CEA $E_{3}^{[s,t]}$}\label{Fig.3}
\end{figure}

  We consider the next cases (see Fig. \ref{Fig.3}):

 \textbf{Case 1.} If $ 0\leq s <1 $ then we have
\begin {itemize}
 \item[(a)]   $E_{3}$ when $ 0\leq t <1 $;
 \item[(b)]   $E_{8}$ when $ 1\leq t <3 $;
 \item[(c)]   $E_{9}$ when $ 3\leq t <6 $;
 \item[(d)]   $E_{8}$ when $ t \geq 6 $;
 \end {itemize}

 \textbf{Case 2.} If $ 1\leq s <2 $ then we have
 \begin {itemize}
 \item[(a)]   $E_{1}$ when $ 1\leq t <2 $;
 \item[(b)]   $E_{6}$ when $ t \geq 2 $;
\end {itemize}

\textbf{Case 3.} If $ 2\leq s <3 $ then we have
\begin {itemize}
 \item[(a)]   $E_{2}$ when $ 2\leq t <3 $;
 \item[(b)]   $E_{8}$ when $ 3\leq t <6 $;
 \item[(c)]   $E_{9}$ when $ t \geq 6 $;
 \end {itemize}

\textbf{Case 4.} If $ 3\leq s <4 $ then we have
\begin {itemize}
 \item[(a)]   $E_{6}$ when $ 3\leq t <6 $;
 \item[(b)]   $E_{12}$ when $ t \geq 6 $;
 \end {itemize}

\textbf{Case 5.} If $ s=4 $ then we have
\begin {itemize}
 \item[(a)]   $E_{4}$ when $ 4\leq t <6 $;
 \item[(b)]   $E_{10}$ when $ t \geq 6 $;
 \end {itemize}

 \textbf{Case 6.} If $ 4< s \leq 5 $ then we have
\begin {itemize}
 \item[(a)]   $E_{5}$ when $ 4< t <6 $;
 \item[(b)]   $E_{11}$ when $ t \geq 6 $;
 \end {itemize}

 \textbf{Case 7.} If $ 5< s <5,7 $ then we have
\begin {itemize}
 \item[(a)]   $E_{8}$ when $ 5< t <6 $;
 \item[(b)]   $E_{9}$ when $ t \geq 6 $;
 \end {itemize}

 \textbf{Case 8.} If $ s=5,7 $ then we have
\begin {itemize}
 \item[(a)]   $E_{5}$ when $ 5,7\leq t <6 $;
 \item[(b)]   $E_{11}$ when $ t \geq 6 $;
 \end {itemize}

 \textbf{Case 9.} If $  s>5,7 $ then we have   $E_{7}$ when $ t>5,7 $.

\end{ex}

\end{document}